\documentclass[12 pt]{article}
\usepackage[pdftex]{graphicx}
\usepackage{amsmath}
\usepackage{mathtools}
\usepackage{amsthm}
\usepackage{amsfonts}
\usepackage{amssymb}
\usepackage[margin=1.0in]{geometry}
\usepackage{tikz}
\usepackage{enumitem}
\usepackage{bm}
\usepackage{multicol}

\newtheorem{theorem}{Theorem}[section]
\newtheorem*{theorem*}{Theorem}
\newtheorem{cor}[theorem]{Corollary}
\newtheorem{lemma}[theorem]{Lemma}
\newtheorem{prop}[theorem]{Proposition}
\newtheorem{que}[theorem]{Question}

\theoremstyle{definition}
\newtheorem{defin}[theorem]{Definition}
\newtheorem{fact}[theorem]{Fact}

\theoremstyle{remark}
\newtheorem*{rem}{Remark}

\newtheorem{exa}[theorem]{Example}

\newcommand{\flim}[1]{\mathrm{Flim}(#1)}
\newcommand{\age}[1]{\mathrm{Age}(#1)}

\newcommand{\fr}{Fra\"iss\'e }
\renewcommand{\phi}{\varphi}

\newcommand{\emb}[1]{\mathrm{Emb}(#1)}
\newcommand{\aut}[1]{\mathrm{Aut}(#1)}

\newcommand{\im}[1]{\mathrm{Im}(#1)}

\newcommand{\hatf}{\,\hat{\rule{-0.5ex}{1.5ex}\smash{f}}}
\newcommand{\tildef}{\,\tilde{\rule{-0.5ex}{1.5ex}\smash{f}}}

\newcommand{\tildei}[1]{\tilde{\imath}_{#1}}
\newcommand{\tp}[1]{\mathrm{tp}(#1)}

\begin{document}
\title{Big Ramsey degrees and topological dynamics}
\author{Andy Zucker}
\maketitle

\begin{abstract}
We consider \fr structures whose objects have finite big Ramsey degree and ask what consequences this has for the dynamics of the automorphism group. Motivated by a theorem of D.\ Devlin about the partition properties of the rationals, we define the notion of a \emph{big Ramsey structure}, a single structure which codes the big Ramsey degrees of a given \fr structure. This in turn leads to the definition of a \emph{completion flow}; we show that if a \fr structure admits a big Ramsey structure, then the automorphism group admits a unique universal completion flow. We also discuss the problem of when big Ramsey structures exist and explore connections to the notion of oscillation stability defined by Kechris, Pestov, and Todor\v{c}evi\'c \cite{KPT}.
\let\thefootnote\relax
\let\thefootnote\relax
\let\thefootnote\relax\footnote{The author was partially supported by NSF Grant no.\ DGE 1252522.}
\end{abstract}

\section{Introduction, main definitions, and statements of theorems}

Consider the statement of Ramsey's theorem: for every $k, r < \omega$, we have
\begin{align*}
\omega\rightarrow (\omega)^k_r
\end{align*}
This ``arrow notation'' is shorthand for the following statement: for any coloring $\gamma: [\omega]^k\rightarrow r$, there is an infinite $S\subseteq \omega$ so that $|\gamma``([S]^k)| = 1$. We often say that $S$ is \emph{monochromatic} for $\gamma$. Ramsey's theorem can be generalized in many directions. Erd\H{o}s and Rado \cite{ErdRad} considered coloring finite tuples from larger cardinals while demanding larger monochromatic sets. Galvin and Prikry \cite{GalPr}, and later Ellentuck \cite{Ell}, considered suitably definable colorings of the \emph{infinite} subsets of $\omega$.

The generalization that we will consider in this paper is that of \emph{structural} Ramsey theory. As a warmup, we will consider the structure $\langle \mathbb{Q}, \leq\rangle$ of rationals and their linear order. In what follows we will just write $\mathbb{Q}$. Let us call a subset $S\subseteq \mathbb{Q}$ a \emph{dense linear order}, or DLO for short, if the ordering on $S$ inherited from $\mathbb{Q}$ is dense and contains no maximum or minimum. Note that we do \emph{not} require that $S$ is dense in $\mathbb{Q}$, nor do we require that $S$ is unbounded in $\mathbb{Q}$. Let us write
\begin{align*}
\mathbb{Q}\rightarrow (\mathbb{Q})^k_r
\end{align*}
for the statement that for any coloring $\gamma: [\mathbb{Q}]^k\rightarrow r$, there is a DLO $S\subseteq \mathbb{Q}$ which is monochromatic for $\gamma$. While this is a natural sounding generalization of Ramsey's theorem, it has the downside of being false. Consider a coloring $\gamma: [\mathbb{Q}]^2\rightarrow 2$ produced as follows: fix an enumeration $\mathbb{Q} = \{q_n: n< \omega\}$, which we can consider as a new \emph{enumeration order} on top of the standard \emph{rational order}. Now given a pair $\{q, r\}$ of rationals, we can ask whether the enumeration order on $q$ and $r$ agrees or disagrees with the rational order, and we set $\gamma(\{q, r\})$ accordingly. It is not hard to see that no DLO subset $S\subseteq \mathbb{Q}$ is monochromatic for $\gamma$, and therefore $\mathbb{Q}\not\rightarrow (\mathbb{Q})^2_2$.

Remarkably, this counterexample is in a strong sense the worst possible. Galvin \cite{Gal} proved that for every $r < \omega$, the following statement holds:
\begin{align*}
\mathbb{Q}\rightarrow (\mathbb{Q})^2_{r, 2}
\end{align*}
This means that for every coloring $\gamma: [\mathbb{Q}]^2\rightarrow r$, there is a DLO subset $S\subseteq \mathbb{Q}$ so that $|\gamma``([S]^2)| = 2$; we call such an $S$ \emph{$2$-chromatic} for $\gamma$. The proof of this statement uses a Ramsey theorem for trees proved by Milliken \cite{Mil}, which in turn is a generalization of the Halpern-L\"auchli theorem. Denis Devlin \cite{Dev} in his thesis pushed Galvin's theorem as far as possible. For each $k< \omega$, he found the exact number $T_k$ so that for any number of colors $r < \omega$, we have 
\begin{align*}
\mathbb{Q}\rightarrow (\mathbb{Q})^k_{r, T_k} \text{ and } \mathbb{Q}\not\rightarrow (\mathbb{Q})^k_{T_k, T_{k}-1}
\end{align*}
So for any coloring $\gamma: [\mathbb{Q}]^k\rightarrow r$, there is some DLO $S\subseteq \mathbb{Q}$ which is $T_k$-chromatic for $\gamma$, but there is a coloring $\delta: [\mathbb{Q}]^k\rightarrow T_k$ so that each color class is \emph{unavoidable}, i.e.\ each color class meets every DLO subset of $\mathbb{Q}$. We say that $T_k$ is the \emph{big Ramsey degree} of any $k$-element substructure of $\mathbb{Q}$. The sequence $T_k$ is given by the \emph{odd tangent numbers}, the sequence of numbers which describes the Taylor series for the tangent function.
\vspace{3 mm}

Let us consider $k = 3$. The strategy we used to construct an unavoidable $2$-coloring of $[\mathbb{Q}]^2$ also works to produce an unavoidable $6$-coloring of $[\mathbb{Q}]^3$. However, it turns out that $T_3 = 16$. We need to work a bit harder to produce an unavoidable $16$-coloring, but the strategy is similar; we introduce some extra relational symbols $\vec{R}$ on top of $\mathbb{Q}$ and then color a triple of rationals based on the expanded structure it receives. The structure $\langle \mathbb{Q}, <, \vec{R}\rangle$ not only works to produce an unavoidable $16$-coloring of $[\mathbb{Q}]^3$, but enjoys a much stronger property. For any $k\geq 2$, let $\gamma_k$ be the map sending a $k$-tuple to the structure inherited from $\langle \mathbb{Q}, <, \vec{R}\rangle$. Then $\gamma_k$ is an unavoidable $T_k$-coloring. We call $\langle \mathbb{Q}, <, \vec{R}\rangle$ a \emph{big Ramsey structure} for $\mathbb{Q}$; these will be the central objects of study in this paper.
\vspace{3 mm}

In order to state the formal definitions in full generality, we need some background on first-order structures. A \emph{relational language} $L = \{R_i: i\in I\}$ is a set of relational symbols; each symbol $R_i$ comes with a finite arity $n_i$. All languages in this paper will be relational. Given a language $L$, an \emph{$L$-structure} $\mathbf{A} = \langle A, R_i^\mathbf{A}\rangle$ is a set $A$ along with an interpretation $R_i^\mathbf{A}\subseteq A^{n_i}$ of each symbol in $L$. We will use boldface for structures and lightface for the underlying set unless otherwise specified. If $\mathbf{A}$ and $\mathbf{B}$ are $L$-structures, an \emph{embedding} $f: \mathbf{A}\rightarrow \mathbf{B}$ is any injective map $f: A\rightarrow B$ so that for each $i\in I$ and each $n_i$-tuple $a_0,...,a_{n_i-1}\in A$, we have 
\begin{align*}
R_i^\mathbf{A}(a_0,...,a_{n_i-1})\Leftrightarrow R_i^\mathbf{B}(f(a_0),...,f(a_{n_i-1})).
\end{align*}
Write $\emb{\mathbf{A}, \mathbf{B}}$ for the set of embeddings from $\mathbf{A}$ to $\mathbf{B}$, and write $\mathbf{A}\leq \mathbf{B}$ if $\emb{\mathbf{A}, \mathbf{B}}\neq \emptyset$. If $A\subseteq B$, then we write $\mathbf{A}\subseteq \mathbf{B}$ if the inclusion map is an embedding. An \emph{isomorphism} is a bijective embedding, and an \emph{automorphism} is an isomorphism from a structure to itself. We write $\emb{\mathbf{A}}$ for $\emb{\mathbf{A}, \mathbf{A}}$, and we write $\aut{\mathbf{A}}$ for the group of automorphisms of $\mathbf{A}$. A structure is \emph{finite} or \emph{countable} if the underlying set is, and we write $|\mathbf{A}|:= |A|$.

If $\mathbf{K}$ is a countable $L$-structure, we write $\age{\mathbf{K}} := \{\mathbf{A}\leq \mathbf{K}: \mathbf{A} \text{ is finite}\}$. A countable structure $\mathbf{K}$ is called a \emph{\fr structure} if for any finite $\mathbf{A}\subseteq \mathbf{K}$ and embedding $f:\mathbf{A}\rightarrow \mathbf{K}$, there is $g\in \aut{\mathbf{K}}$ with $g|_{\mathbf{A}} = f$. Two facts are worth pointing out, both due to \fr \cite{Fr}. First, if $\mathbf{K}$ is a \fr structure, then $\mathcal{K}:= \age{\mathbf{K}}$ is a \emph{\fr class}; this is any class of $L$-structures with the following four properties.
\begin{enumerate}
\item
$\mathcal{K}$ contains only finite structures, contains structures of arbitrarily large finite cardinality, and is closed under isomorphism.
\item
$\mathcal{K}$ has the \emph{Hereditary Property} (HP): if $\mathbf{B}\in \mathcal{K}$ and $\mathbf{A}\subseteq \mathbf{B}$, then $\mathbf{A}\in \mathcal{K}$.
\item
$\mathcal{K}$ has the \emph{Joint Embedding Property} (JEP): if $\mathbf{A}, \mathbf{B}\in \mathcal{K}$, then there is $\mathbf{C}\in \mathcal{K}$ which embeds both $\mathbf{A}$ and $\mathbf{B}$.
\item
$\mathcal{K}$ has the \emph{Amalgamation Property} (AP): if $\mathbf{A}, \mathbf{B}, \mathbf{C}\in \mathcal{K}$ and $f: \mathbf{A}\rightarrow \mathbf{B}$ and $g: \mathbf{A}\rightarrow \mathbf{C}$ are embeddings, there is $\mathbf{D}\in \mathcal{K}$ and embeddings $r: \mathbf{B}\rightarrow \mathbf{D}$ and $s:\mathbf{C}\rightarrow\mathbf{D}$ with $r\circ f = s\circ g$.
\end{enumerate}
Second, if $\mathcal{K}$ is a \fr class, there is up to isomorphism a unique \fr structure $\mathbf{K}$ with $\age{\mathbf{K}} = \mathcal{K}$. We call $\mathbf{K}$ the \emph{\fr limit} of $\mathcal{K}$ and write $\mathbf{K} = \flim{\mathcal{K}}$.
\vspace{3 mm}

\begin{defin}\mbox{}
\label{BigRamsey}
\begin{enumerate}
\item
Let $\mathbf{K}$ be a \fr structure with $\mathcal{K} = \age{\mathbf{K}}$. Let $\mathbf{A}\in \mathcal{K}$, and let $r, \ell < \omega$. The statement
\begin{align*}
\mathbf{K}\rightarrow (\mathbf{K})^\mathbf{A}_{r, \ell}
\end{align*}
says that for any coloring $\gamma: \emb{\mathbf{A}, \mathbf{K}}\rightarrow r$, there is $\eta\in \emb{\mathbf{K}}$ with \newline $|\{\gamma(\eta\circ f): f\in \emb{\mathbf{A}, \mathbf{K}}\}|\leq \ell$.
\item
With notation as above, we say that $\mathbf{A}$ has \emph{big Ramsey degree} $\ell< \omega$ if $\ell$ is least so that for every $r > \ell$, we have $\mathbf{K}\rightarrow (\mathbf{K})^\mathbf{A}_{r, \ell}$.
\item
We say that $\mathbf{A}$ has \emph{finite big Ramsey degree} if $\mathbf{A}$ has big Ramsey degree $\ell$ for some $\ell  <\omega$.
\end{enumerate}
\end{defin}
For example, we saw earlier that if $\mathcal{K}$ is the class of finite linear orders and $\mathbf{K} = \mathbb{Q}$, then the big Ramsey degree of the $k$-element linear order is the $k$-th odd tangent number. We also built a ``big Ramsey structure'' which somehow captured the correct big Ramsey degree for every finite substructure of $\mathbb{Q}$. Definition \ref{BRStructure} is one of the central definitions of this paper.
\vspace{3 mm}

\begin{defin}
Let $A$ be a set, and let $\mathbf{B}$ be an $L$-structure. Let $f: A\rightarrow B$ be injective. Then $\mathbf{B}\!\cdot \!f$ is the unique $L$-structure with underlying set $A$ so that $f: \mathbf{B}\!\cdot\! f\rightarrow \mathbf{B}$ is an embedding.
\end{defin}
\vspace{0 mm}

\begin{defin}
\label{BRStructure}
Let $\mathbf{K}$ be a \fr $L$-structure with $\mathcal{K} = \age{\mathbf{K}}$. We say that $\mathbf{K}$ \emph{admits a big Ramsey structure} if there is a language $L'\supseteq L$ and an $L'$-structure $\mathbf{K}'$ so that the following all hold.
\begin{enumerate}
\item
$\mathbf{K}'|_L = \mathbf{K}$.
\item
Every $\mathbf{A}\in \mathcal{K}$ has finitely many $L'$-expansions to a structure $\mathbf{A}'\in \age{\mathbf{K}'}$; denote the set of expansions by $\mathbf{K}'(\mathbf{A})$.
\item
Every $\mathbf{A}\in \mathcal{K}$ has big Ramsey degree $|\mathbf{K}'(\mathbf{A})|$.
\item
The function $\gamma: \emb{\mathbf{A}, \mathbf{K}}\rightarrow \mathbf{K}'(\mathbf{A})$ given by $\gamma(f) = \mathbf{K}'\cdot f$ witnesses the fact that the big Ramsey degree of $\mathbf{A}$ is not less than $|\mathbf{K}'(\mathbf{A})|$.
\end{enumerate}
Call a structure $\mathbf{K}'$ satisfying (1)-(4) a \emph{big Ramsey structure} for $\mathbf{K}$.
\end{defin}
\vspace{3 mm}

We will see that big Ramsey structures give rise to various new dynamical objects. In particular, we will define the notion of a \emph{completion flow} and show that big Ramsey structures imply the existence of a \emph{universal} completion flow. In order to define these objects, we need some background in topological dynamics.
\vspace{3 mm}

Let $G$ be a Hausdorff topological group. A (right) \emph{$G$-flow} is a compact Hausdorff space $X$ equipped with a continuous right action $a: X\times G\rightarrow X$. When the action $a$ is understood, we will write $x\cdot g$ or $xg$ instead of $a(x, g)$. A $G$-ambit is a pair $(X, x_0)$ where $X$ is a $G$-flow and $x_0\in X$ is a point with dense orbit. A \emph{pre-ambit} is any $G$-flow containing a point with dense orbit. In the literature, these are often called \emph{point transitive}, but the name ``pre-ambit'' will be more suggestive going forward.

If $X$ and $Y$ are $G$-flows, a \emph{$G$-map} or \emph{map of flows} is a continuous map $\phi: X\rightarrow Y$ so that for any $x\in X$ and $g\in G$, we have $\phi(xg) = \phi(x)g$. If $(X, x_0)$ and $(Y, y_0)$ are ambits, then a \emph{map of ambits} is a $G$-map $\phi: X\rightarrow Y$ so that $\phi(x_0) = y_0$. While there may be many $G$-maps from $X$ to $Y$, there is at most one map of ambits from $(X, x_0)$ to $(Y, y_0)$. Also notice that any map of ambits is surjective.

 Each topological group comes equipped with several compatible uniform structures, including the \emph{left uniformity}. A typical member of the left uniformity is a set of the form $\{(g, h)\in G\times G: g^{-1}h\in V\}$, where $V$ is an open symmetric neighborhood of the identity. This will be the only uniform structure we place on $G$ for now. A net $\{g_i: i\in I\}$ from $G$ is called \emph{Cauchy} if for every symmetric open $V\subseteq G$ containing the identity, there is $i_0\in I$ so that for every $i, j > i_0$, we have $g_i^{-1}g_j\in V$. Just like metric spaces, every uniform space admits a completion. We let $\widehat{G}$ denote the left completion of $G$; $\widehat{G}$ then enjoys the structure of a topological semigroup. Now if $X$ is a $G$-flow, there is a unique extension of the $G$-action to $\widehat{G}$ making $X\times \widehat{G}\rightarrow X$ continuous. 
\vspace{3 mm}

\begin{defin}\mbox{}
\label{CompletionAmbit}
\begin{enumerate}
\item
Let $X$ be a $G$-flow. A point $x\in X$ is called a \emph{completion point} if for every $\eta\in \widehat{G}$, $x\cdot \eta$ has dense orbit. The flow $X$ is called a \emph{completion flow} if $X$ contains a completion point; the ambit $(X, x_0)$ is called a \emph{completion ambit} if $x_0$ is a completion point of $X$.

\item
A completion flow $X$ is called \emph{universal} if for any other completion flow $Y$, there is a surjective $G$-map $\phi: X\rightarrow Y$.
\end{enumerate}
\end{defin}
\vspace{3 mm}

If $\mathbf{K} = \flim{\mathcal{K}}$ is a \fr structure, then $G = \aut{\mathbf{K}}$ becomes a topological group when endowed with the topology of pointwise convergence. A typical open neighborhood of the identity is a set of the form $\{g\in G: \forall a\in A (g(a) = a)\}$, where $A\subseteq K$ is some finite subset. The left completion $\widehat{G}$ is the semigroup $\emb{\mathbf{K}}$, which we also endow with the topology of pointwise convergence. 

For $G = \aut{\mathbf{K}}$, there are two types of $G$-flows we will often consider. The first is a space of colorings. If $\mathbf{A}\in \mathcal{K}$, then $\widehat{G}$ acts on the set $\emb{\mathbf{A}, \mathbf{K}}$ on the left by composition. Now let $r < \omega$, and form the space $X = r^{\emb{\mathbf{A}, \mathbf{K}}}$ of $r$-colorings of $\emb{\mathbf{A}, \mathbf{K}}$ endowed with the product topology. If $\gamma \in X$ and $\eta\in \widehat{G}$, then we define $\gamma\cdot \eta$ by setting $\gamma\cdot \eta(f) = \gamma(\eta\cdot f)$. 

The second type of $G$-flow is a space of structures. Let $L$ be a language (not necessarily the same language that $\mathbf{K}$ is in). Let $X$ the space of $L$-structures on the set $K$. We give $X$ the \emph{logic topology}, where a typical open set of structures is of the form $\{x\in X: x\cdot f = \mathbf{A}\}$, where $\mathbf{A}$ is some finite $L$-structure and $f: A\rightarrow K$ is some injection. If $x\in X$ and $\eta\in \widehat{G}$, we define $x\cdot \eta$ as follows. Let $R\in L$ be an $n$-ary relation symbol. If $a_0,...,a_{n-1}\in K$, then we set $R^{x\cdot \eta}(a_0,...,a_{n-1})$ iff $R^x(\eta(a_0),...,\eta(a_{n-1}))$ holds. In this level of generality, $X$ may not be compact, but the various subspaces of $X$ we discuss will always be compact. When speaking about structures in a space of structures, we will often break with our notational convention of using boldface. 
\vspace{3 mm}

\begin{defin}
\label{BRAmbit}
Let $\mathbf{K}$ be a \fr $L$-structure, and let $\mathbf{K}'$ be a big Ramsey structure for $\mathbf{K}$ in a language $L'\supseteq L$. Form the space $X$ of $L'$-structures on $K$, and let $X_{\mathbf{K}'} = \overline{\mathbf{K}'\cdot G}$ be the orbit closure of $\mathbf{K}'\in X$. We call any ambit $(Z, z_0)$ isomorphic to $(X_{\mathbf{K}'}, \mathbf{K}')$ a \emph{big Ramsey ambit}, and we call the underlying flow $Z$ a \emph{big Ramsey flow}.
\end{defin}
\vspace{3 mm}

If $\mathbf{K}'$ is a big Ramsey structure and $X_{\mathbf{K}'} = \overline{\mathbf{K}'\cdot G}$, then $\mathbf{K}'$ is a completion point of $X_{\mathbf{K}'}$ (see Proposition \ref{BRCompletion}), so every big Ramsey flow is a completion flow.

Before stating the main theorems of this paper, a discussion of the history and motivation is in order. A very fruitful direction of research for the past 15 years has been the interaction between the combinatorics of \fr structures and the dynamical properties of their automorphism groups. One of the first efforts in this direction was due to Pestov \cite{Pe}, where he proved that the group $\aut{\mathbb{Q}}$ is \emph{extremely amenable}, meaning that every flow admits a fixed point. His proof makes crucial use of the finite version of Ramsey's theorem. Kechris, Pestov, and Todor\v{c}evi\'c \cite{KPT} then showed that in a very strong sense, the fact that $\aut{\mathbb{Q}}$ is extremely amenable is actually equivalent to the finite Ramsey theorem. If $\mathbf{K} = \flim{\mathcal{K}}$ is a \fr structure, $\mathbf{A}\leq \mathbf{B}\in \mathcal{K}$, and $\ell < r <\omega$, write
\begin{align}
\label{SmallRamsey}
\mathbf{K}\rightarrow (\mathbf{B})^\mathbf{A}_{r, \ell}
\end{align}
for the statement that for every coloring $\gamma: \emb{\mathbf{A}, \mathbf{K}}\rightarrow r$, there is $s\in \emb{\mathbf{B}, \mathbf{K}}$ so that $|\{\gamma(s\circ f): f\in \emb{\mathbf{A}, \mathbf{B}}\}|\leq \ell$. We say that $\mathbf{A}$ has \emph{small Ramsey degree} $\ell < \omega$ if $\ell$ is least so that for every $\mathbf{B}\in \mathcal{K}$ with $\mathbf{A}\leq \mathbf{B}$ and every $r > \ell$, the statement (\ref{SmallRamsey}) holds. We say that $\mathbf{A}$ is a \emph{Ramsey object} if $\mathbf{A}$ has small Ramsey degree $\ell = 1$. The first major theorem in \cite{KPT} states that $G = \aut{\mathbf{K}}$ is extremely amenable iff every $\mathbf{A}\in \mathcal{K}$ is a Ramsey object. 

If $G$ is a topological group, a $G$-flow $X$ is called \emph{minimal} if every orbit is dense. A minimal flow is called \emph{universal} if it admits a $G$-map to any other minimal flow. Every topological group admits a \emph{universal minimal flow}, or UMF,  which is unique up to isomorphism. The second major result of \cite{KPT}, which was generalized to the form given here in \cite{NVT0}, provided a construction of the UMF of $G =\aut{\mathbf{K}}$ in several cases. Very roughly, if $\mathbf{K}$ is a \fr $L$-structure, this construction proceeds by exhibiting a \fr expansion $\mathbf{K}'$ in some languange $L'\supseteq L$ satisfying several conditions. If this can be done, then form the space $X$ of $L'$-structures on $K$, and let $Y = \overline{\mathbf{K}'\cdot G}$ be the orbit closure. Then $Y$ is the universal minimal flow of $G$. One feature of this construction is that if it can be done, then each $\mathbf{A}\in \mathcal{K}$ must have finite small Ramsey degree.

The current author in \cite{Z} then showed that the construction of universal minimal flows in \cite{KPT} and \cite{NVT0} is in a sense the only one possible and characterized exactly when it could be carried out. It turns out that if $\mathbf{K} = \flim{\mathcal{K}}$ is a \fr structure and each $\mathbf{A}\in \mathcal{K}$ has finite small Ramsey degree, then the construction from \cite{KPT} and \cite{NVT0} is possible. Furthermore, this characterizes exactly when $G = \aut{\mathbf{K}}$ has metrizable UMF. 

The main goal of this paper is to attempt to carry out a similar analysis in regards to \emph{big} Ramsey degree. If $\mathbf{K} = \flim{\mathcal{K}}$ and $\mathcal{K}$ contains all objects of finite big Ramsey degree, does this correspond to $\aut{\mathbf{K}}$ having a ``nicely described'' metrizable universal object in some category of dynamical systems? And can this universal object be described as a space of structures? 

Our main theorem is the following.
\vspace{3 mm}

\begin{theorem}
\label{BRUCA}
Let $\mathbf{K}$ be a \fr structure which admits a big Ramsey structure, and let $G = \aut{\mathbf{K}}$. Then any big Ramsey flow is a universal completion flow, and any two universal completion flows are isomorphic.
\end{theorem}
\vspace{3 mm}

The proof of Theorem \ref{BRUCA} introduces new techniques in abstract topological dynamics which seem interesting in their own right. In particular, we will define the notion of \emph{strong maps} between pre-ambits and see that the category of pre-ambits and strong maps enjoys a rich structure. However, there is still much we do not know about completion flows. In particular, the following fundamental question remains open: if $G$ is a topological group, does $G$ admit a unique universal completion flow?

Theorem \ref{BRUCA} makes it important to know when a \fr structure $\mathbf{K}$ admits a big Ramsey structure. An obvious necessary condition is that every $\mathbf{A}\in \mathcal{K}$ have finite big Ramsey degree. Whether or not this is sufficient seems to be a difficult question; we will discuss the key difficulties in section \ref{SetsAndColorings}. Rather strangely, when $G = \aut{\mathbf{K}}$ is \emph{Roelcke precompact}, we can say quite a bit about what the big Ramsey ambit should look like if it exists. In particular, we can describe exactly what the age of any big Ramsey structure should be, and we can show that any big Ramsey flow must contain a generic orbit. 

The possible gap between having finite big Ramsey degrees and admitting a big Ramsey structure also suggests a useful weakening of the notion of \emph{oscillation stability} defined in \cite{KPT}. We will precisely define this notion later, but if $\mathbf{K} = \flim{\mathcal{K}}$ and $G = \aut{\mathbf{K}}$, then $G$ is oscillation stabile iff each object in $\mathcal{K}$ has big Ramsey degree one. By a deep theorem of Hjorth \cite{Hj}, this notion is vacuous for Polish groups; no Polish group can be oscillation stable. 

The weakening we propose is as follows; call a topological group $G$ \emph{completely amenable} if $G$ admits no non-trivial completion flows. So any completely amenable group is also extremely amenable. The example of $G = \aut{\mathbb{Q}}$ shows that the converse does not hold. We will see that if a topological group $G$ is oscillation stable, then $G$ is completely amenable. However, the converse is not known. The difficulty is similar in nature to the difficulty in showing that for a structure $\mathbf{K} = \flim{\mathcal{K}}$, having finite big Ramsey degrees implies that $\mathbf{K}$ admits a big Ramsey structure. Therefore the question of whether any Polish group is completely amenable has content.
\vspace{3 mm}

This paper is organized as follows. Section \ref{BackgroundSection} contains background on compact left-topological semigroups, Samuel compactifications, and automorphism groups of \fr structures. The construction of the Samuel compactification of $\aut{\mathbf{K}}$ given in section \ref{FraisseSG} is essential for the rest of the paper. Section \ref{LiftsSection} defines the \emph{lift} of an ambit and shows that this lift only depends on the underlying flow. Section \ref{SetsAndColorings} contains a variety of combinatorial results about unavoidable colorings. Section \ref{ProofSection} uses the ideas of sections \ref{LiftsSection} and \ref{SetsAndColorings} to prove Theorem \ref{BRUCA}. Section \ref{ExampleSection} contains examples of groups with universal completion flows

After section \ref{ExampleSection}, the nature of the paper changes considerably. Section \ref{ExpansionSection} addresses the question of whether or not big Ramsey structures exist. While the question in general remains open, a number of results are proven which describe what the big Ramsey flow must look like should it exist. 

The final section, Section \ref{QuestionSection}, discusses connections to oscillation stability and gathers a list of relevant open questions.

There is also an appendix which contains the proof of Theorem \ref{LevelsBR}, an intuitive result with a cumbersome proof.

\subsection*{Acknowledgements}
I thank Clinton Conley and James Cummings for helpful discussions, and I thank Lionel Nguyen Van Th\'e and Jean Larson for pointing me to several examples. I also thank Stevo Todor\v{c}evi\'c for suggestions on an earlier version of this paper.
\vspace{5 mm}

\section{Background}
\label{BackgroundSection}

\subsection{Compact left-topological semigroups}

An excellent reference on compact left-topological semigroups is the first two chapters of the book by Hindman and Strauss \cite{HS}. Readers should note however the left-right switch between that reference and the presentation here.

Let $S$ be a semigroup. If $x\in S$, let $\lambda_x: S\rightarrow S$ and $\rho_x: S\rightarrow S$ denote the left and right multiplication maps, respectively. A non-empty semigroup $S$ is a \emph{compact left-topolgical semigroup} if $S$ is also a compact Hausdorff space so that for every $x\in S$, the map $\lambda_x$ is continuous. Given $x\in S$ and a subset $T\subseteq S$, we often write $xT := \{xy: y\in T\}$ and $Tx := \{yx: y\in T\}$. A \emph{right ideal} (respectively \emph{left ideal}) of a semigroup $S$ is a subset $M\subseteq S$ so that for every $x\in M$, we have $xS\subseteq M$ (respectively $Sx\subseteq M$). An \emph{idempotent} is any element $x\in S$ with $xx = x$.

We will freely use the following facts throughout the paper. In the following, $S$ denotes a compact left-topological semigroup.
\vspace{3 mm}

\begin{fact}\mbox{}
\label{SemiFacts}
\begin{itemize}
\item
(Ellis-Numakura) $S$ contains an idempotent.

\item
If $u\in S$ is idempotent and $x\in uS$, then $ux = x$. If $x\in Su$, then $xu = x$.

\item
Every right ideal $M\subseteq S$ contains a closed right ideal; namely if $x\in M$, then $xS = \lambda_x(S)$ is closed and a right ideal. Then by Zorn's lemma, every right ideal contains a minimal right ideal which must be closed. Every minimal right ideal is a compact left-topological semigroup, so contains an idempotent.

\item
$S$ contains minimal left ideals. If $M$ is a minimal right ideal and $x\in M$, then $Sx$ is a minimal left ideal. The intersection of any minimal right ideal and any minimal left ideal is a group, hence contains exactly one idempotent.

\item
If $M$ and $N$ are minimal right ideals and $x\in M$, then there is $y\in N$ with $yx\in N$ an idempotent.
\end{itemize}
\end{fact}

\subsection{The Samuel compactification}

An excellent reference on the Samuel compactification is Samuel's original paper \cite{Sa}. For a more modern presentation focused on topological groups, see Uspenskij \cite{Usp}. For more on the greatest ambit and topological dynamics, see Auslander \cite{Aus}.

Let $(X, \mathcal{U})$ be a Hausdorff uniform space. The \emph{Samuel compactification} of $X$ is a compact Hausdorff space $S(X)$ along with a uniformly continuous map $i: X\rightarrow S(X)$ satisfying the followin universal property: if $Y$ is any compact Hausdorff space and $f: X\rightarrow Y$ is uniformly continuous, then there is a unique map $\tilde{f}: S(X)\rightarrow Y$ making the diagram commute.
\begin{center}
\begin{tikzpicture}[node distance=3cm, auto]
\node (X) {$X$};
\node (Y) [right of=X] {$Y$};
\node (SX) [node distance=2cm, above of=X] {$S(X)$};
\draw[->] (X) to node [swap] {$f$} (Y);
\draw[->] (X) to node {$i$} (SX);
\draw[->] (SX) to node {$\tilde{f}$} (Y);
\end{tikzpicture}
\end{center}
The map $i$ is always an embedding, so we often identify $X$ with its image in $S(X)$ and use the term ``Samuel compactification'' to just refer to the space $S(X)$.

We will be primarily interested in the Samuel compactification of a topological group $G$ equipped with its left uniformity. In this case, there is a quick, albeit uninformative, construction of $S(G)$. Let $\{(X_i, x_i): i\in I\}$ be a set of $G$-ambits. Form the product $\prod_i X_i$; this is a $G$-flow with $G$ acting on each coordinate. Now let $z_0 = (x_i)_{i\in I}$, and let $Z = \overline{z_0G}\subseteq \prod_i X_i$ be the orbit closure. Then $(Z, z_0)$ is an ambit, and the projection onto coordinate $i$ is a map of ambits from $(Z, z_0)$ to $(X_i, x_i)$. Suppose that we started with a set of ambits containing a representative of each isomorphism type of ambit; this is possible because there are only set-many compact Hausdorff spaces with a dense set of size at most $|G|$, and only set-many possible $G$-flow structures to place on each one. Then the ambit $(Z, z_0) := (S(G), 1_G)$ admits a map of ambits onto any other $G$-ambit. The point $1_G$ with dense orbit in $S(G)$ is just the identity of $G$ upon identifying $G$ with a subset of $S(G)$. The ambit $(S(G), 1_G)$ is often called the \emph{greatest ambit} for $G$.

The universal property allows us to endow $S(G)$ with the structure of a compact left-topological semigroup. If $x\in S(G)$, then $(\overline{xG}, x)$ is an ambit, so let $\lambda_x: S(G)\rightarrow \overline{xG}$ be the unique map of ambits. The notation is deliberately suggestive; we define this to be the left multiplication map, i.e.\ for $y\in S(G)$, we set $xy:= \lambda_x(y)$. It is routine to check that this multiplication is associative, and left multiplication is continuous by definition.

We can also use the universal property to have $S(G)$ act on $G$-flows. If $X$ is a $G$-flow and $x\in X$, then $(\overline{xG}, x)$ is an ambit; let $\phi_x: S(G)\rightarrow X$ be the unique map of ambits. If $p\in S(G)$, we then set $xp:= \phi_x(p)$. This extended action behaves nicely with the semigroup structure on $S(G)$; if $x\in X$ and $p, q\in S(G)$, we then have $x(pq) = (xp)q$.

We end this subsection with a simple lemma about subflows of $S(G)$. If $Y\subseteq S(G)$ is a pre-ambit, it makes sense to ask about the semigroup properties of points $y\in Y$ with $\overline{yG} = Y$. In particular, if some point with dense orbit is an idempotent, this has consequences for the dynamics of $Y$
\vspace{3 mm}

\begin{prop}
\label{PreambitSG}
Let $Y\subseteq S(G)$ be a pre-ambit, and assume there is an idempotent $u\in Y$ with dense orbit. Then if $\phi: Y\rightarrow S(G)$ is a $G$-map, then $\phi = \lambda_{\phi(u)}$.
\end{prop}

\begin{proof}
First note that $Y = \overline{uG} = uS(G)$. Fix $x\in Y$. Then $ux = x$ by Fact \ref{SemiFacts}. So $\phi(x) = \phi(ux) = \phi(u)x$.
\end{proof}

\subsection{\fr structures and Samuel compactifications}
\label{FraisseSG}

For a more detailed exposition of both the notational conventions developed here and the construction of $S(G)$ for $G = \aut{\mathbf{K}}$, see \cite{Z}.

For this subsection, fix a \fr structure $\mathbf{K} = \flim{\mathcal{K}}$ with $G = \aut{\mathbf{K}}$. An \emph{exhaustion} of $\mathbf{K}$ is a sequence $\{\mathbf{A}_n: n<\omega\}$ with $\mathbf{A}_n\subseteq \mathbf{A}_{n+1}\subseteq \mathbf{K}$, $\mathbf{A}_n\in \mathcal{K}$, and $\mathbf{K} = \bigcup_n \mathbf{A}_n$. When we write $\mathbf{K} = \bigcup_n \mathbf{A}_n$, we will assume that $\{\mathbf{A}_n: n< \omega\}$ is an exhaustion unless otherwise specified. We also assume $\mathbf{A}_0 = \emptyset$. For the rest of this subsection, fix an exhaustion of $\mathbf{K}$. Let $i_m: \mathbf{A}_m\rightarrow \mathbf{K}$ denote the inclusion embedding.

Let $m\leq n$. As a shorthand notation, write $H_m:= \emb{\mathbf{A}_m, \mathbf{K}}$ and $H_m^n := \emb{\mathbf{A}_m, \mathbf{A}_n}$. In particular, we have $H_m = \bigcup_{n\geq m} H^n_m$. Also, for any $m< \omega$, we have $H_0 = H_0^m = \{\emptyset\}$. Notice that $G$ acts on $H_m$ on the left by postcomposition. Since $\mathbf{K}$ is a \fr structure, this action is transitive. We will often write $g|_m$ for $g\cdot i_m$. 

Each $f\in H_m^n$ gives rise to a dual map $\hatf: H_n\rightarrow H_m$ given by precomposition, i.e.\ if $s\in H_n$, we have $\hatf(s) = s\circ f$. The notation is slightly imprecise, since the range of $f$ must be specified to know the domain of the dual map, but this will typically be clear from context. The following basic facts record the properties of dual maps we will use.
\vspace{3 mm}

\begin{prop}\mbox{}
\label{Amalgamation}
\begin{enumerate}
\item
For $f\in H^n_m$, the dual map $\hatf: H_n\rightarrow H_m$ is surjective.
\item
For every $f\in H_m^n$, there is $h\in H_n$ with $h\circ f = i^N_m$.
\end{enumerate}
\end{prop}
\vspace{0 mm}

\begin{proof}
For the first item, fix $f\in H_m^n$, and let $h\in H_m$. Find $g\in G$ with $g\cdot f = h$. Then by definition, we have $\hatf{g|_n} = g|_n\circ f = g\cdot f = h$.
\vspace{3 mm}

For the second item, fix $f\in H_m^n$. Find $g\in G$ with $g\cdot f = i_m$. Then we have $g|_n\circ f = g\cdot f = i_m$.
\end{proof}
\vspace{3 mm}

Our main goal for this subsection is to construct $S(G)$ and write down an explicit formula for its semigroup multiplication. The construction will require ``glueing together'' spaces of ultrafilters, so we briefly discuss ultrafilters and the \v{C}ech-Stone compactification.

Let $X$ be a set. A collection $p\subseteq \mathcal{P}(X)$ is called an \emph{ultrafilter} if the following conditions are met.
\begin{itemize}
\item
$X\in p$ and $\emptyset\not\in p$
\item
If $A\subseteq B\subseteq X$ and $A\in p$, then $B\in p$.

\item
If $A, B\subseteq X$ and $A, B\in p$, then $A\cap B\in p$.

\item
If $A\subseteq X$, then either $A\in p$ or $X\setminus A \in p$.
\end{itemize}
If $p\subseteq \mathcal{P}(X)$ only satisfies the first three items, then $p$ is called a filter. By Zorn's lemma, maximal filters exist, and the maximal filters on $X$ coincide with the ultrafilters on $X$.

Write $\beta X$ for the collection of all ultrafilters on $X$. We endow $\beta X$ with a compact Hausdorff, zero-dimensional topology where the typical clopen subset of $\beta X$ has the form $\overline{A}:= \{p\in \beta X: A\in p\}$, where $A\subseteq X$. We can identify $X$ as a subset of $\beta X$ by identifying each $x\in X$ with the ultrafilter $\{A\subseteq X: x\in A\}$. Viewing $X$ as a discrete space, the inclusion $i: X\rightarrow \beta X$ is called the \emph{\v{C}ech-Stone compactification} of $X$. It satisfies the same universal property as the Samuel compactification of $X$ when $X$ is given the discrete uniformity. To be explicit, if $f: X\rightarrow Y$ is \emph{any} map from $X$ to a compact Hausdorff space $Y$, there is a unique continuous map $\tilde{f}$ making the following diagram commute.
\begin{center}
\begin{tikzpicture}[node distance=3cm, auto]
\node (X) {$X$};
\node (Y) [right of=X] {$Y$};
\node (SX) [node distance=2cm, above of=X] {$\beta X$};
\draw[->] (X) to node [swap] {$f$} (Y);
\draw[->] (X) to node {$i$} (SX);
\draw[->] (SX) to node {$\tilde{f}$} (Y);
\end{tikzpicture}
\end{center}

If $Y\subseteq \beta X$ is a closed subspace, then the collection $\mathcal{F}_Y:= \{A\subseteq X: Y\subseteq \overline{A}\}$ is a filter on $X$. Conversely, if $\mathcal{H}$ is a filter on $X$, then $\tilde{\mathcal{H}} := \{p\in \beta X: \mathcal{F}\subseteq p\}$ is a closed subspace. Given a closed $Y\subseteq \beta X$, we have $\tilde{\mathcal{F}}_Y = Y$, and if $\mathcal{H}$ is a filter on $X$, we have $\mathcal{F}_{\tilde{\mathcal{H}}} = \mathcal{H}$. The following fact will be essential going forward. A proof can be found in \cite{HS}.
\vspace{3 mm}

\begin{fact}
\label{MetrizableBeta}
If $Y\subseteq \beta X$ is a closed, metrizable subspace, then $Y$ is finite.
\end{fact}
\vspace{3 mm}

We now turn to the construction of $S(G)$. The main idea of the construction is to view the sets $H_m$ as discrete spaces and consider their \v{C}ech-Stone compactifications. Let $f\in H^n_m$. Then the dual map $\hatf: H_n\rightarrow H_m\subseteq \beta H_m$ extends uniquely to a continuous map $\tildef: \beta H_n\rightarrow \beta H_m$. Given $p\in \beta H_n$, we often write $p\cdot f := \tildef(p)$. 

Form the inverse limit $\varprojlim \beta H_n$ of the spaces $\beta H_n$ along the maps $\tildei{m}: \beta H_n\rightarrow \beta H_m$. This is also a zero-dimensional compact Hausdorff space. Let $\pi_m: \varprojlim \beta H_n\rightarrow \beta H_m$ be the projection onto the $m$-th coordinate; then a typical clopen neighborhood is of the form $\{\alpha\in \varprojlim \beta H_n: S\in \pi_m(\alpha)\}$, where $m< \omega$ and $S\subseteq H_m$. We have the following fact due to Pestov \cite{Pe}.
\vspace{3 mm}

\begin{fact}
$S(G)\cong \varprojlim \beta H_n$
\end{fact}
\vspace{3 mm}

From now on, we will identify $S(G)$ with $\varprojlim \beta H_n$. We now proceed to exhibit the right $G$-action on $S(G)$ which makes $(S(G), 1_G)$ the greatest ambit. This might seem unnatural at first; after all, the left $G$-action on each $H_n$ extends to a left $G$-action on $\beta H_n$, giving us a left $G$-action on $S(G)$. However, this left $G$-action isn't continuous when $G$ is given its pointwise convergence topology.  The right action we describe doesn't operate on any one $\beta H_m$; the various ``levels'' of the inverse limit will interact with each other in our definition. 

For $\alpha\in S(G)$, we often write $\alpha(m)$ for $\pi_m(\alpha)$. 
\vspace{3 mm}

\begin{defin}
Given $\alpha\in S(G)$ and $g\in G$, we define $\alpha g\in S(G)$ by setting $\alpha g(m) = \alpha(n)\cdot g|_m$, where $n\geq m$ is large enough so that $g|_m \in H_m^n$. More explicitly, if $\alpha\in S(G)$, $g\in G$, $m< \omega$, and $S\subseteq H_m$, we have
\begin{align*}
S\in \alpha g(m) \Leftrightarrow \{x\in H_n: x\circ g|_m\in S\}\in \alpha(n)
\end{align*}
where $n\geq m$ is suitably large. 
\end{defin}
\vspace{3 mm}

From the definition, we see that this right action is continuous. We embed the left completion $\widehat{G}$ into $S(G)$ by identifying $\eta\in \widehat{G}$ with the element $i(\eta)$ of $S(G)$ with $i(\eta)(m) = \eta|_m\in H_m\subseteq \beta H_m$. By regarding $G$ as a subset of $S(G)$, we now have the following fact from \cite{Z}
\vspace{3 mm}

\begin{fact}
$(S(G), 1_G)$ equipped with the above right action is the greatest $G$-ambit.
\end{fact}
\vspace{3 mm}

We can also write explicitly the left-topological semigroup structure on $S(G)$. Fix $\alpha\in S(G)$. If $f\in H_m$ and $n, N\geq m$ are both suitably large, we see that $\alpha(n)\cdot f = \alpha(N)\cdot f$, and we simply write $\alpha \cdot f$. We write $\rho_f$ for the map $\alpha\rightarrow \alpha\cdot f$. Notice that $\rho_f$ is continuous.

The map from $H_m$ to $\beta H_m$ sending $f$ to $\alpha\cdot f$ has a unique continuous extention. We denote this extension by $\lambda_\alpha^m$, while often writing $\lambda_\alpha^m(p) := \alpha \cdot p$. More explicitly, if $p\in \beta H_m$ and $S\subseteq H_m$, we have 
\begin{align*}
S\in \alpha\cdot p \Leftrightarrow \{f\in H_m: S\in \alpha\cdot f\}\in p
\end{align*}
Given $\alpha\in S(G)$ and $S\subseteq H_m$, a useful shorthand is to put $\alpha^{-1}(S) := \{f\in H_m: S\in \alpha\cdot f\}$. Then we can write $S\in \alpha\cdot p$ iff $\alpha^{-1}(S)\in p$. 

The notation $\lambda_\alpha^m$ is deliberately suggestive, as this will be the left multiplication by $\alpha$ ``restricted to level $m$.'' If $\alpha$ and $\gamma$ in $\varprojlim \beta H_n$, we define $\alpha\cdot \gamma$ by setting $(\alpha\cdot \gamma)(m) = \alpha\cdot (\gamma(m))$, and we write $\lambda_\alpha$ for the map semding $\gamma$ to $\alpha\cdot \gamma$. To check that this operation is the semigroup operation discussed abstractly in the previous subsection, it suffices to prove the following proposition.
\vspace{3 mm}

\begin{prop}
\label{SemigroupLaw}
Fix $\alpha\in S(G)$. Then $\lambda_\alpha$ is a $G$-map with $\lambda_\alpha(1_G) = \alpha$.
\end{prop}

\begin{proof}
By definition $\lambda_\alpha$ is continuous. Fix $\gamma\in S(G)$ and $g\in G$. Fix a net $g_i\in G$ with $g_i\rightarrow \gamma$. As the map $\rho_{g|_m}$ is continuous, we have $g_i\circ g|_m\rightarrow \gamma\cdot g|_m$. Then as $\lambda_\alpha^m$ is continuous, we have $\alpha\cdot (g_i\circ g|_m) \rightarrow \alpha\cdot (\gamma\cdot g|_m)$. But $\alpha\cdot (g_i\circ g|_m) = (\alpha\cdot g_i)\cdot g|_m$, so once more by continuity of the maps $\lambda_\alpha^m$ and $\rho_{g|_m}$, we have $(\alpha\cdot g_i)\cdot g|_m\rightarrow (\alpha\cdot \gamma)\cdot g|_m$. It follows that $(\alpha\cdot \gamma)\cdot g = \alpha \cdot (\gamma\cdot g)$ as desired. That $\lambda_\alpha(1_G) = \alpha$ is clear. 
\end{proof}
\vspace{3 mm}

We end this subsection by recording some basic facts about metrizable subspaces and subflows of $S(G)$. If $Y\subseteq S(G)$ is metrizable, then so is $\pi_m''(Y)$ for every $m< \omega$. Fact \ref{MetrizableBeta} immediately yields the following. 
\vspace{3 mm}

\begin{fact}
\label{MetrizableSG}
$Y\subseteq S(G)$ is metrizable iff $\pi_m''(Y)$ is finite for each $m< \omega$.
\end{fact}
\vspace{3 mm}

In the same vein as Proposition \ref{PreambitSG}, we can ask about pre-ambits $Y\subseteq S(G)$ containing an idempotent with dense orbit. In the case $G = \aut{\mathbf{K}}$, we can say more.
\vspace{3 mm}

\begin{prop}
\label{EndoIso}
Let $Y\subseteq S(G)$ be a metrizable pre-ambit, and assume there is an idempotent $u\in Y$ with dense orbit. If $\phi: Y\rightarrow Y$ is a surjective $G$-map, then $\phi$ is an isomorphism.
\end{prop}

\begin{proof}
By Proposition \ref{PreambitSG}, we have $\phi = \lambda_{\phi(u)}$. Then the map $\lambda_{\phi(u)}^m: \pi_m''(Y)\rightarrow \pi_m''(Y)$ must be surjective. Since $\pi_m''(Y)$ is finite, $\lambda_{\phi(u)}^m$ must be a bijection. It follows that $\phi$ is a bijection, hence an isomorphism.
\end{proof}

\section{Lifts of ambits and pre-ambits}
\label{LiftsSection}

This section contains most of the dynamical content needed in the proof of Theorem \ref{BRUCA}. We actually develop more than what we will need and in greater generality, since the techniques presented here seem interesting in their own right. Throughout this seciton, let $G$ be a Hausdorff topological group, and let $\widehat{G}$ denote its left completion.
\vspace{3 mm}

\begin{defin}\mbox{}
\begin{enumerate}
\item
Let $X$ be a pre-ambit. The \emph{ambit set} of $X$ is the set $\mathcal{A}(X):= \{x\in X: \overline{xG} = X\}$.
\item
Let $X$ and $Y$ be pre-ambits. A $G$-map $\phi: X\rightarrow Y$ is called \emph{strong} if for any $x\in X$, we have $x\in \mathcal{A}(X)\Leftrightarrow \phi(x)\in \mathcal{A}(Y)$
\end{enumerate}
\end{defin}
\vspace{3 mm}

Whenever we refer to strong maps, we will always assume that the domain and range are pre-ambits. It is worth pointing out that any strong map between pre-ambits must be surjective. Notice that if $\phi: X\rightarrow Y$ is any surjective $G$-map between pre-ambits, then $x\in \mathcal{A}(X)$ implies that $\phi(x)\in \mathcal{A}(Y)$. It is the reverse implication that makes strong maps useful. The following easy proposition hints at why this notion will be useful going forward.
\vspace{3 mm}

\begin{prop}
\label{PreimageStrongCompletion}
Let $Y$ be a completion flow, and let $\phi: X\rightarrow Y$ be a strong $G$-map. Then $X$ is a completion flow. 
\end{prop}

\begin{proof}
Let $y\in Y$ be a completion point. Pick $x\in X$ with $\phi(x) = y$. We claim that $x$ is a completion point. Let $\eta\in \widehat{G}$. Then $\phi(x\eta) = \phi(x)\eta = y\eta\in \mathcal{A}(Y)$, so $x\eta\in \mathcal{A}(X)$.
\end{proof}
\vspace{3 mm}

The following definition will be our main source of strong maps. Recall that if $(X, x_0)$ is an ambit, we write $\phi_{x_0}: S(G)\rightarrow X$ for the unique map of ambits from $S(G)$.
\vspace{3 mm}

\begin{defin}
\label{LiftAmbit}
Let $(X, x_0)$ be an ambit.
\begin{enumerate}
\item
The \emph{fixed point semigroup} of $(X, x_0)$ is $S_{x_0} := \phi_{x_0}^{-1}(\{x_0\}) = \{p\in S(G): x_0\cdot p = x_0\}$. It is a closed subsemigroup of $S(G)$, hence a compact, left-topological semigroup in its own right.
\item
A \emph{lift} of $(X, x_0)$ is any subflow $Y\subseteq S(G)$ which is minimal subject to the property that $Y\cap S_{x_0}\neq \emptyset$.
\end{enumerate}
\end{defin}
\vspace{3 mm}
Notice that since $S_{x_0}$ is compact, Zorn's lemma ensures that any ambit admits a lift. The next lemma records some simple observations about lifts.
\vspace{3 mm}

\begin{lemma}
\label{LiftLemmas}
Let $(X, x_0)$ be an ambit, and let $Y\subseteq S(G)$ be a lift of $(X, x_0)$.
\begin{enumerate}
\item
$\phi_{x_0}|_Y: Y\rightarrow X$ is a strong $G$-map.
\item
$Y\cap S_{x_0}$ is a minimal right ideal of $S_{x_0}$. 
\end{enumerate}
\end{lemma}

\begin{proof}
For item (1), first note that $\phi_{x_0}|_Y$ is surjective since $x_0\in \phi_{x_0}''(Y)$. Let $y\in Y$ be a point with $\phi_{x_0}(y)\in \mathcal{A}(X)$. Then there is $p\in S(G)$ with $\phi_{x_0}(y)p = x_0$. Then $yp\in S_{x_0}$, so in particular $\overline{yG}\cap S_{x_0} \neq \emptyset$. By  the minimality property of lifts, we must have $\overline{yG} = Y$, so $y\in \mathcal{A}(Y)$ is a transitive point.
\vspace{3 mm}

For item (2), certainly $Y\cap S_{x_0}$ is a right ideal of $S_{x_0}$, so suppose $M\subseteq Y\cap S_{x_0}$   is a minimal right ideal of $S_{x_0}$, and let $y\in M$. By the minimality property of lifts, we must have $yS(G) = Y$. Suppose $p\in S(G)\setminus S_{x_0}$. Then $x_0yp = x_0p\neq x_0$, so $yp\not\in S_{x_0}$. It follows that $Y\cap S_{x_0} = y\cdot S_{x_0}\subseteq M$, so $M = Y\cap S_{x_0}$.
\end{proof}
\vspace{3 mm}

The next two propositions show that the choice of lift doesn't matter. The first shows that any two lifts are isomorphic, and the second limits the nature of $G$-maps between lifts.

\begin{prop}
\label{LiftsIso}
Let $(X, x_0)$ be an ambit, and let $Y_0, Y_1\subseteq S(G)$ be two lifts of $(X, x_0)$. Then $Y_0$ and $Y_1$ are isomorphic over $(X, x_0)$, i.e.\ there is an isomorphism $\psi: Y_0\rightarrow Y_1$ so that the following diagram commutes.
\begin{center}
\begin{tikzpicture}[node distance=3cm, auto]
\node (Y_0) {$Y_0$};
\node (Y_1) [right of=Y_0] {$Y_1$};
\node (X) [node distance=1.5cm, right of=Y_0, below of=Y_0] {$X$};
\draw[->] (Y_0) to node {$\psi$} (Y_1);
\draw[->] (Y_0) to node [swap] {$\phi_{x_0}$} (X);
\draw[->] (Y_1) to node {$\phi_{x_0}$} (X);
\end{tikzpicture}
\end{center}
\end{prop}

\begin{proof}
Write $M_i = Y_i\cap S_{x_0}$. Then each $M_i$ is a minimal right ideal of $S_{x_0}$ by item (2) of Lemma \ref{LiftLemmas}. Let $v\in M_1$ be an idempotent. Then the left multiplication $\lambda_v: M_0\rightarrow M_1$ is an isomorphism of right ideals of $S_{x_0}$. Using Fact \ref{SemiFacts}, let $u\in M_0$ be an idempotent in the same minimal left ideal of $S_{x_0}$ as $v$. Then $uv = u$ and $vu = v$. It follows that $\psi := \lambda_v: Y_0\rightarrow Y_1$ is an isomorphism with inverse $\lambda_u$. To check that the diagram commutes, let $y\in Y_0$. Then $y = up$ for some $p\in S(G)$. Then $\phi_{x_0}\circ \lambda_v(up) = \phi_{x_0}(vp) = x_0p = \phi_{x_0}(up)$.
\end{proof}
\vspace{3 mm}

\begin{prop}
\label{LiftsEndo}
Let $(X, x_0)$ be an ambit, and let $Y_0, Y_1\subseteq S(G)$ be two lifts of $(X, x_0)$. Let $\psi: Y_0\rightarrow Y_1$ be a surjective $G$-map making the diagram from Proposition \ref{LiftsIso} commute. Then $\psi$ is an isomorphism.
\end{prop}

\begin{proof}
Once again, write $M_i = Y_i\cap S_{x_0}$. Each $M_i$ is a minimal right ideal of $S_{x_0}$.  Let $u\in M_0$ be an idempotent. Then $\psi = \lambda_{\psi(u)}$ by Proposition \ref{PreambitSG}. As $\psi(u)\in M_1$, use Fact \ref{SemiFacts} to find $v\in M_0$ with $v\psi(u)\in M_0$ an idempotent. Since $v\psi(u)\in M_0$, we have $Y_0 = v\psi(u)S(G)$, so in particular, $\lambda_v\circ \lambda_{\psi(u)}$ is the identity map on $Y_0$. It follows that $\psi$ is an isomorphism.
\end{proof}

We have shown that the lift of any ambit is canonical in the sense of Proposition \ref{LiftsIso}. Remarkably, the lift of any pre-ambit is also canonical; if $X$ is a preambit and $x_0, x_1\in \mathcal{A}(X)$, then the lifts of the ambits $(X, x_0)$ and $(X, x_1)$ will be isomorphic as well. The rest of this section is spent proving this fact; it will not be needed in the proof of Theorem \ref{BRUCA}, but these ideas will be used in section \ref{RPLiftSection}.
\vspace{3 mm}

If $X$ and $Y$ are pre-ambits and $\psi: Y\rightarrow X$ is a strong map, we call $\psi$ a \emph{strong extension} of $X$. A strong extension $\psi_0: Y_0\rightarrow X$ is called \emph{universal} if given any other strong extension $\psi_1: Y_1\rightarrow X$, there is a strong map $\phi: Y_0\rightarrow Y_1$ with $\psi_1\circ \phi = \psi_0$. 

\begin{center}
\begin{tikzpicture}[node distance=3cm, auto]
\node (Y_0) {$Y_0$};
\node (Y_1) [right of=Y_0] {$Y_1$};
\node (X) [node distance=1.5cm, right of=Y_0, below of=Y_0] {$X$};
\draw[->] (Y_0) to node {$\phi$} (Y_1);
\draw[->] (Y_0) to node [swap] {$\psi_0$} (X);
\draw[->] (Y_1) to node {$\psi_1$} (X);
\end{tikzpicture}
\end{center}

If $i< 2$ and $\psi_i: Y_i\rightarrow X$ are two strong extensions, we say that $\psi_0$ and $\psi_1$ are \emph{isomorphic over $X$} if there is an isomorphism $\phi: Y_0\rightarrow Y_1$ with $\psi_1\circ \phi = \psi_0$.
\vspace{3 mm}

\begin{theorem}
\label{PreambitLift}
Let $X$ be a pre-ambit. Then there is a universal strong extension $\psi_X: S(X)\rightarrow X$. Any two universal strong extensions are isomorphic over $X$.
\end{theorem}
\vspace{3 mm}

We call the pre-ambit $S(X)$ given by Theorem \ref{PreambitLift} the \emph{universal strong extension} of $X$. We have previously used the notaion $S(X)$ to denote the Samuel compactification, but the context should always be clear. The next two propositions will prove Theorem \ref{PreambitLift}. The first produces a universal strong extension of any pre-ambit, and the second shows uniqueness.

The following notation will be useful. If $(X, x_0)$ is an ambit, we let $S(X, x_0)\subseteq S(G)$ be a lift of $X$. 
\vspace{3 mm}

\begin{prop}
Let $X$ be a pre-ambit, and let $x_0\in \mathcal{A}(X)$. Then $\phi_{x_0}|_{S(X, x_0)}: S(X, x_0)\rightarrow X$ is a universal strong extension.
\end{prop}

\begin{proof}
Let $Y$ be a pre-ambit, and fix a strong map $\psi: Y\rightarrow X$. Pick $y_0\in Y$ with $\psi(y_0) = x_0$. Then $y_0\in \mathcal{A}(Y)$, and $S_{y_0}\subseteq S_{x_0}$. So we may assume that $S(X, x_0)\subseteq S(Y, y_0)$.

We will show that $S(X, x_0) = S(Y, y_0)$. Using the minimality property of lifts, it suffices to show that $\phi_{y_0}''(S(X, x_0))\cap \mathcal{A}(Y)\neq \emptyset$. Let $p\in S(X, x_0)\cap S_{x_0}$, towards showing that $\phi_{y_0}(p)\in\mathcal{A}(Y)$. First note that $\psi\circ \phi_{y_0} = \phi_{x_0}$. Then $\phi_{x_0}(p) = \psi\circ \phi_{y_0}(p) = \psi(y_0p) = x_0$. Since $\psi$ is strong, we have $y_0p = \phi_{y_0}(p)\in A(Y)$. 

Putting everything together, we now have $S(X, x_0) = S(Y, y_0)$, and $\phi_{y_0}: S(X, x_0)\rightarrow Y$ is a strong map with $\psi\circ \phi_{y_0} = \phi_{x_0}$.
\end{proof}
\vspace{3 mm}

\begin{prop}
Let $X$ be a pre-ambit. Then any two universal strong extensions are isomorphic over $X$. 
\end{prop}

\begin{proof}
Let $x_0\in \mathcal{A}(X)$, and form $\phi_{x_0}: S(X, x_0)\rightarrow X$. Suppose that $\psi: Y\rightarrow X$ were another universal strong extension. By using the universal property of each map and composing, we obtain a strong map $\alpha: S(X, x_0)\rightarrow S(X, x_0)$ with $\phi_{x_0} = \phi_{x_0}\circ \alpha$. It is enough to note that $\alpha$ must be an isomorphism, and this follows by Proposition \ref{LiftsEndo}
\end{proof}

\section{Sets and colorings}
\label{SetsAndColorings}

This section covers the combinatorial content needed going forward. Fix a \fr structure $\mathbf{K} = \flim{\mathcal{K}} = \bigcup_n \mathbf{A}_n$. Set $G = \aut{\mathbf{K}}$, and let $\widehat{G} = \emb{\mathbf{K}}$ be the left completion. Recall that we have set $H_m := \emb{\mathbf{A}_m, \mathbf{K}}$ and $H_m^n := \emb{\mathbf{A}_m, \mathbf{A}_n}$

In the introduction, we saw that if $k, m < \omega$, then $k^{H_m}$, the space of $k$-colorings of $H_m$, has a natural $G$-flow structure. The case $k = 2$ deserves a special mention, as we will freely identify $S\subseteq H_m$ with its characteristic function $\chi_S\in 2^{H_m}$. If $S\in H_m$ and $\eta\in \widehat{G}$, we have $\chi_S\cdot \eta = \chi_{\eta^{-1}(S)}$, where $\eta^{-1}(S) := \{f\in H_m: \eta\cdot f\in S\}$. It is helpful to think of $\eta$ as defining a ``copy'' of $\mathbf{K}$ inside $\mathbf{K}$, and the operation $S\rightarrow \eta^{-1}(S)$ ``zooms in'' on that copy. For example, if $S = \{\eta\cdot f: f\in H_m\}$, then $\eta^{-1}(S) = H_m$.

Recall from subsection \ref{FraisseSG} that if $S\subseteq H_m$ and $\alpha\in S(G)$, we set $\alpha^{-1}(S) = \{f\in H_m: S\in \alpha\cdot f\}$.  
\vspace{3 mm}

\begin{prop}
If $\alpha\in S(G)$ and $S\subseteq H_m$, we have $\chi_S\cdot \alpha = \chi_{\alpha^{-1}(S)}$. 
\end{prop}

\begin{proof}
Fix $f\in H_m$, and let $g_i\in G$ be a net with $g_i\rightarrow \alpha$. Then $g_i\cdot f\rightarrow \alpha\cdot f$. It follows that eventually $S\in \alpha\cdot f$ iff $g_i\cdot f\in S$, so eventually $f\in \alpha^{-1}(S)$ iff $\chi_S\cdot g_i(f) = 1$. But eventually $\chi_S\cdot g_i(f) = \chi_S \cdot \alpha(f)$.
\end{proof}
\vspace{3 mm}

We now define the key combinatorial notions we will need going forward. Many of these notions describe properties that a subset $S\subseteq H_m$ may or may not have. In the spirit of identifying $\mathcal{P}(H_m)$ with $2^{H_m}$, we will say that $\gamma\in 2^{H_m}$ has one of these properties iff $\{f\in H_m: \gamma(f) = 1\}$ has the property.
\vspace{3 mm} 

\begin{defin}\mbox{}
\label{ComboDef}
\begin{enumerate}
\item
A set $S\subseteq H_m$ is called \emph{large} if for some $\eta\in \widehat{G}$, we have $\eta^{-1}(S) = H_m$.

\item
A set $S\subseteq H_m$ is called \emph{unavoidable} if $H_m\setminus S$ is not large. Equivalently, $S$ is large if for every $\eta\in \widehat{G}$, we have $\eta^{-1}(S)\neq \emptyset$.

\item
A set $S\subseteq H_m$ is called \emph{somewhere unavoidable} if for some $\eta\in \widehat{G}$, we have $\eta^{-1}(S)$ unavoidable.

\item
A set $S\subseteq H_m$ is called \emph{scattered} if $S$ is not somewhere unavoidable.

\item
Fix $k\leq r < \omega$, and let $\gamma: H_m\rightarrow r$ be a coloring. We call $\gamma$ an \emph{unavoidable $k$-coloring} if $|\im{\gamma}| = k$ and for each $i< r$, we have $\gamma^{-1}(\{i\})\subseteq H_m$ either empty or unavoidable. We call $\gamma$ an \emph{unavoidable coloring} if $\gamma$ is an unavoidable $k$-coloring for some $k\leq r$.
\end{enumerate}
\end{defin}
\vspace{3 mm}

\begin{rem}
The term ``scattered'' comes from the theory of linear orders. A linear order $X$ is called \emph{scattered} if $X$ does not embed $\mathbb{Q}$. We can think of points as embeddings of the singleton linear order; if $\mathbf{A}_1\subseteq \mathbb{Q}$ is a single point, then $S\subseteq H_1$ is scattered in the traditional sense iff it is scattered in our sense.
\end{rem}
\vspace{3 mm}

The next few propositions investigate how objects with these properties behave under images and preimages of the dual maps defined in subsection \ref{FraisseSG}. Recall that if $f\in H_m^n$, we set $\hatf: H_n\rightarrow H_m$ given by $\hatf(s) = s\circ f$. If $S\subseteq H_m$, we often write $f^n(S) := \hatf^{-1}(S)$, or just $f(S)$ if there is no ambiguity.
\vspace{3 mm}

\begin{lemma}
Let $m\leq n<\omega$. Fix $f\in H_m^n$. If $S\subseteq H_m$ and $\eta\in \widehat{G}$, then $\eta^{-1}(f(S)) = f(\eta^{-1}(S))$.
\end{lemma}

\begin{proof}
\begin{align*}
\\[-15 mm]
\eta^{-1}(f(S)) &= \{s\in H_n: \eta\cdot s\in f(S)\}\\
&= \{s\in H_n: (\eta\cdot s)\circ f\in S\}\\
&= \{s\in H_n: \eta\cdot (s\circ f) \in S\}\\
&= \{s\in H_n: s\circ f\in \eta^{-1}(S)\}\\
&= f(\eta^{-1}(S))\qedhere
\end{align*}
\end{proof}
\vspace{3 mm}

\begin{lemma}
\label{ComboLevels}
Let $m\leq n <\omega$. Fix $f\in H_m^n$. 
\begin{enumerate}
\item
Let $S\subseteq H_m$. Then $S$ has any of the properties from Definition \ref{ComboDef} (1)-(4) iff $f(S)$ has the corresponding property. 

\item
Let $S\subseteq H_n$. If $S$ has any of the properties from Definition \ref{ComboDef} (1)-(3), then $\hatf(S)$ also has the corresponding property. If $\hatf(S)$ is scattered, then $S$ is scattered.
\end{enumerate}
\end{lemma}

\begin{proof}
First let $S\subseteq H_m$ be large, and fix $\eta\in \widehat{G}$ with $\eta^{-1}(S) = H_m$. Then $\eta^{-1}(f(S)) = f(\eta^{-1}(S)) = f(H_m) = H_n$. Conversely, assume $f(S)$ is large, and find $\eta\in \widehat{G}$ with $\eta^{-1}(f(S)) = H_n$. But then $f(\eta^{-1}(S)) = H_n$. Since $\hatf$ is surjective, we must have $\eta^{-1}(S) = H_m$, so $S$ is large.
\vspace{3 mm}

The statement in part (1) of the lemma for ``unavoidable'' follows immediately. 
\vspace{3 mm}

Now let $S\subseteq H_m$ be somewhere unavoidable. Find $\eta\in \widehat{G}$ with $\eta^{-1}(S)$ unavoidable. Then $\eta^{-1}(f(S)) = f(\eta^{-1}(S))$ is unavoidable, so $f(S)$ is somewhere unavoidable. Conversely, assume $f(S)$ is somewhere unavoidable, and find $\eta\in \widehat{G}$ with $\eta^{-1}(f(S))$ unavoidable. But $\eta^{-1}(f(S)) = f(\eta^{-1}(S))$, so $\eta^{-1}(S)$ is unavoidable, and $S$ is somewhere unavoidable.
\vspace{3 mm}

The statement in part (1) of the lemma for ``scattered'' follows immediately.
\vspace{3 mm}

Now let $S\subseteq H_n$. Note that the properties (1)-(3) from Definition \ref{ComboDef} are closed upwards, while being scattered is closed downwards. Noting that $f(\hatf(S))\supseteq S$, we are done by part (1) of the lemma.
\end{proof}
\vspace{3 mm}

\begin{cor}
Let $m\leq n < \omega$, let $k< \omega$, and let $\gamma: H_m\rightarrow k$ be an unavoidable coloring. If $f\in H_m^n$, then $\gamma\circ \hatf: H_n\rightarrow k$ is also  unavoidable.
\end{cor}
\vspace{3 mm}

\begin{cor}
Let $m\leq n< \omega$. If $\mathbf{A}_m$ and $\mathbf{A}_n$ have finite big Ramsey degrees $R_m$ and $R_n$, respectively, then $R_m \leq R_n$.
\end{cor}
\vspace{3 mm}

We now turn our attention to colorings. We call $\gamma$ a \emph{coloring of $H_m$} if $\gamma \in k^{H_m}$ for some $k< \omega$.
\vspace{3 mm}

\begin{defin}\mbox{}
\label{Refine}
\begin{enumerate}
\item
Let $\gamma$ and $\delta$ be colorings of $H_m$. We say that $\delta$ \emph{refines} $\gamma$ and write $\gamma \leq \delta$ if whenever $f_0, f_1\in H_m$ and $\delta(f_0) = \delta(f_1)$, then $\gamma(f_0) = \gamma(f_1)$. Refinement is a pre-order. We say $\gamma$ and $\delta$ are \emph{equivalent} and write $\gamma \sim \delta$ if $\gamma \leq \delta$ and $\delta \leq \gamma$. Equivalence is an equivalence relation.

\item
Fix $m\leq n < \omega$. Let $\gamma$ be a coloring of $H_m$, and let $\delta$ be a coloring of $H_n$. We say that $\delta$ \emph{strongly refines} $\gamma$ and write $\gamma \ll \delta$ if for every $f\in H_m^n$, we have that $\gamma\circ \hatf \leq \delta$.

\item
Let $\gamma$ and $\delta$ be colorings of $H_m$. A \emph{product coloring} of $\gamma$ and $\delta$ is any coloring $\gamma * \delta$ so that for any $f_0, f_1\in H_m$, we have $(\gamma * \delta)(f_0) = (\gamma * \delta)(f_1)$ iff $\gamma(f_0) = \gamma(f_1)$ and $\delta(f_0) = \delta(f_1)$. It is unique up to equivalence, so we usually call $\gamma * \delta$ \emph{the} product coloring.
\end{enumerate}
\end{defin}
\vspace{3 mm}

\begin{lemma}
\label{StayUnavoidable}
Let $\gamma$ be an unavoidable $k$-coloring of $H_m$. Then for every $\eta\in \widehat{G}$, $\gamma\cdot \eta$ is also an unavoidable $k$-coloring.
\end{lemma}

\begin{proof}
Immediate from the definition.
\end{proof}
\vspace{3 mm}

\begin{lemma}
\label{DynamicRefine}
Let $\gamma$ and $\delta$ be colorings of $H_m$ with $\gamma \leq \delta$. If $\alpha\in S(G)$, then $\gamma\cdot \alpha \leq \delta \cdot \alpha$. If $n\geq m$ and $\phi$ is a coloring of $H_n$ with $\gamma \ll \phi$, then also $\gamma\cdot \alpha \ll \phi\cdot \alpha$.
\end{lemma}

\begin{proof}
It suffices to prove the first part of the lemma. Fix $f_0, f_1\in H_m$. Let $g_i\in G$ be a net with $g_i\rightarrow \alpha$. If $\delta\cdot \alpha(f_0) = \delta\cdot \alpha(f_1)$, then eventually $\delta\cdot g_i(f_0) = \delta\cdot g_i(f_1)$. So eventually $\gamma\cdot g_i(f_0) = \gamma\cdot g_i(f_1)$, and $\gamma\cdot \alpha(f_0) = \gamma\cdot \alpha(f_1)$.
\end{proof}
\vspace{3 mm}

\begin{lemma}
\label{MakeUnavoidable}
Let $\gamma: H_m\rightarrow r$ be a coloring. Then there is $\eta\in \widehat{G}$ with $\gamma\cdot \eta$ unavoidable.
\end{lemma}

\begin{proof}
For each $i \leq r$, we define $\eta_i\in \widehat{G}$ inductively as follows. Set $\eta_0 = 1_G$. If $i < r$ and $\eta_i$ has been determined, consider the coloring $\gamma\cdot \eta_i$. If $(\gamma\cdot \eta_i)^{-1}(\{i\})$ is unavoidable, set $\eta_{i+1} = \eta_i$. Otherwise, find $\zeta\in \widehat{G}$ so that $(\gamma\cdot \eta_i\cdot \zeta)^{-1}(\{i\}) = \emptyset$, and set $\eta_{i+1} = \eta_i\cdot \zeta$. Then $\gamma\cdot \eta_r$ is unavoidable.
\end{proof}
\vspace{3 mm}

When objects in the \fr class $\mathcal{K}$ have finite big Ramsey degrees, unavoidable colorings gain quite a bit of structure. In particular, if $\mathbf{A}_m$ has big Ramsey degree $R_m< \omega$, this means that $R_m$ is the largest number for which there is an unavoidable $R_m$-coloring of $H_m$.
\vspace{3 mm}

\begin{prop}
\label{ColorRefine}
Suppose that $\mathbf{A}_m$ has finite big Ramsey degree $R_m < \omega$. Let $\gamma$ be any coloring of $H_m$, and let $\delta$ be an unavoidable $R_m$-coloring of $H_m$. Then there is $\eta\in \widehat{G}$ so that $\gamma\cdot \eta \leq \delta\cdot \eta$.
\end{prop}

\begin{proof}
Form the product coloring $\gamma * \delta$. Use Lemma \ref{MakeUnavoidable} to find $\eta\in \widehat{G}$ with $(\gamma * \delta)\cdot \eta$ unavoidable. By Lemma \ref{StayUnavoidable}, $\delta\cdot \eta$ is an unavoidable $R_m$-coloring, and by Lemma \ref{DynamicRefine}, \newline $(\gamma * \delta)\cdot \eta \geq \delta\cdot \eta$. Since $R_m$ is the big Ramsey degree, we must have $(\gamma * \delta)\cdot \eta \sim \delta\cdot \eta$. Then $\gamma\cdot \eta\leq \delta \cdot \eta$ as desired.
\end{proof}
\vspace{3 mm}

\begin{cor}
\label{StrongRefine}
Let $m\leq n < \omega$. Suppose that $\mathbf{A}_n$ has finite big Ramsey degree $R_n < \omega$. Let $\gamma$ be any coloring of $H_m$, and let $\delta$ be an unavoidable $R_n$-coloring of $H_n$. Then there is $\eta\in \widehat{G}$ with $\gamma\cdot \eta \ll \delta \cdot \eta$.
\end{cor}
\vspace{3 mm}

As promised in the introduction, we now discuss why $\mathcal{K}$ having finite big Ramsey degrees doesn't necessarily show that $\mathbf{K}$ admits a big Ramsey ambit. To illustrate this difficulty, let us consider some key differences between the study of small Ramsey degrees and big Ramsey degrees. Much as lower bounds to big Ramsey degrees are witnessed by unavoidable colorings, the lower bounds to small Ramsey degrees are witnessed by \emph{syndetic} colorings. If $\gamma$ is a $k$-coloring of $H_m$, we call $\gamma$ \emph{syndetic} if for some $n\geq m$ and every $s\in H_n$, we have $|\{\gamma(s\circ f): f\in H_m^n\}| = k$. 

In general, we cannot strenthen Lemma \ref{StayUnavoidable} to say that if $\gamma$ is an unavoidable $k$-coloring of $H_m$ and $\alpha\in S(G)$, then $\gamma\cdot \alpha$ is also an unavoidable $k$-coloring. However, this does hold for syndetic colorings; if $\gamma$ is a syndetic $k$-coloring of $H_m$ and $\alpha\in S(G)$, then $\gamma\cdot \alpha$ is also a syndetic $k$-coloring of $H_m$. 

Now suppose each $\mathbf{A}_m$ has finite big Ramsey degree $R_m$. Using Corollary \ref{StrongRefine}, for any $N < \omega$, we can find unavoidable $R_m$-colorings $\gamma_m$ of $H_m$ for each $m< N$ so that $\gamma_m \ll \gamma_n$ whenever $m\leq n < N$. However, for $\mathbf{K}$ to admit a big Ramsey structure, we need to find unavoidable $R_m$ colorings for \emph{every} $m< \omega$ so that $\gamma_m \ll \gamma_n$ whenever $m\leq n$ (see Theorem \ref{LevelsBR}). If each $\mathbf{A}_m$ has small Ramsey degree $r_m$ and we wish to do the same with syndetic $r_m$-colorings of $H_m$, this is an easy compactness argument precicely because the ``strengthened'' version of Lemma \ref{StayUnavoidable} holds for syndetic colorings. 

Similarly, suppose there is some $m< \omega$ so that $\mathbf{A}_m$ does not have finite small Ramsey degree. This means that for every $k< \omega$, we can find a syndetic $k$-coloring of $H_m$. Using compactness, it is easy to cook up an ``infinite'' syndetic coloring, i.e.\ a partition of $H_m$ into countably many disjoint ``syndetic'' sets. Such a coloring can then be used to show the existence of a non-metrizable minimal $G$-flow. On the other hand, if $\mathbf{A}_m$ does not have finite big Ramsey degree, then for every $k < \omega$, there is an unavoidable $k$-coloring. However, it is not clear that there is a partition of $H_m$ into countably many disjoint unavoidable sets, nor is it clear that such a partition yields the existence of a non-metrizable completion flow.

Even though we cannot strengthen Lemma \ref{StayUnavoidable} in general, there is an important case where we can make such a strenthening.
\vspace{3 mm}

\begin{prop}
\label{RefineUnavoidable}
Let $\gamma$ and $\delta$ be unavoidable $k$ and $\ell$ colorings of $H_m$, respectively, with $\gamma \leq \delta$. Let $\alpha\in S(G)$, and assume that $\delta\cdot \alpha$ is an unavoidable $\ell$-coloring. Then $\gamma\cdot \alpha$ is an unavoidable $k$-coloring.
\end{prop}

\begin{proof}
We may assume that $\gamma$ and $\delta$ have range $k$ and $\ell$, respectively. Let $\psi: \ell\rightarrow k$ be the map so that for $f\in H_m$ and $i < \ell$, we have $\delta(f) = i$ implies $\gamma(f) = \psi(i)$. The map $\psi$ induces a $G$-map $\tilde\psi: \ell^{H_m}\rightarrow k^{H_m}$, where if $\phi\in \ell^{H_m}$ and $f\in H_m$, we have $\tilde\psi(\phi)(f) = \psi(\phi(f))$. Note that $\tilde\psi(\delta) = \gamma$, so since $\tilde\psi$ is a $G$-map, we have $\tilde\psi(\delta\cdot \alpha) = \gamma\cdot \alpha$. Since $\delta\cdot \alpha$ is unavoidable and $\delta\cdot \alpha \geq \gamma\cdot \alpha$, we have that $\gamma\cdot \alpha$ is unavoidable. And since $\psi$ is surjective, we have that $\gamma\cdot \alpha$ is an unavoidable $k$-coloring.
\end{proof}
\vspace{3 mm}

\section{Proof of Theorem \ref{BRUCA}}
\label{ProofSection}

This section culminates with the proof of Theorem \ref{BRUCA}. Along the way, we will need a deeper understanding of the ``level-by-level'' dynamics of spaces of structures. Throughout this section, fix a \fr $L$-structure $\mathbf{K} = \flim{\mathcal{K}} = \bigcup_n \mathbf{A}_n$ which admits a big Ramsey structure $\mathbf{K}'$ in a language $L'\supseteq L$. We let $X_{\mathbf{K}'} = \overline{\mathbf{K}'\cdot G}$ be the orbit closure of $\mathbf{K}'$ in the space of $L'$-structures on $K$. For each $m< \omega$, we let $R_m <\omega$ denote the big Ramsey degree of $\mathbf{A}_m$.

We will often refer to ``colorings'' of $H_m$ whose ranges are finite, but not necessarily contained in $\omega$; all the definitions and theorems from the previous section transfer to this more general notion of coloring in the obvious way. 

The first easy proposition shows that big Ramsey flows are completion flows.

\begin{prop}
\label{BRCompletion}
$\mathbf{K}'\in X_{\mathbf{K}'}$ is a completion point.
\end{prop}

\begin{proof}
Fix $\eta\in \widehat{G}$. We need to find $p\in S(G)$ with $\mathbf{K}'\cdot \eta \cdot p = \mathbf{K}'$. Let $n < \omega$. Since $\mathbf{K}'$ is a big Ramsey structure, we can find $g_n\in G$ with $\mathbf{K}'\cdot \eta \cdot g_n|_n = \mathbf{K}'\cdot i_n$. Let $p\in S(G)$ be a cluster point of the $g_n$. Then $\mathbf{K}'\cdot \eta\cdot p = \mathbf{K}'$ as desired.
\end{proof}
\vspace{3 mm}

Recall that $\pi_m: S(G)\rightarrow \beta H_m$ is the projection onto level $m$. If $Y\subseteq S(G)$ is closed, we write $Y_m := \pi_m''(Y)$. Then we have $Y = \varprojlim Y_m$.
\vspace{3 mm}

\begin{lemma}
\label{CompletionLift}
Let $(X, x_0)$ be a completion ambit, and let $Y\subseteq S(G)$ be a lift of $(X, x_0)$. Then $Y$ is an completion flow, and for every $m< \omega$, we have $|Y_m|\leq R_m$.
\end{lemma}

\begin{proof}
Write $M = Y\cap S_{x_0}$, and fix $y\in M$. Notice that $y$ is a completion point of $Y$ by Proposition \ref{PreimageStrongCompletion} and by item (1) of Lemma \ref{LiftLemmas}. Let $\gamma: H_m\rightarrow Y_m$ be given by $\gamma(f) = y\cdot f$. Suppose $r < \omega$ and that $\delta: Y_m\rightarrow r$ is a continuous surjection. Then since $y$ is a comletion point, $\delta\circ \gamma$ is an unavoidable $r$-coloring, so $r\leq R_m$. In particular, since $Y_m$ is zero-dimensional, we must have $|Y_m|\leq R_m$. 
\end{proof}
\vspace{3 mm}

\begin{rem}
We didn't need the big Ramsey structure to prove Lemma \ref{CompletionLift}. We only needed the big Ramsey degrees to be finite. So whenever $\mathbf{K} = \flim{\mathcal{K}}$ is a \fr structure so that every $\mathbf{A}\in \mathcal{K}$ has finite big Ramsey degree, then every completion flow of $G = \aut{\mathbf{K}}$ is metrizable.
\end{rem}
\vspace{3 mm}

Recall that if $A$ is a set, $\mathbf{B}$ is an $L'$-structure and $f: A\rightarrow B$ is injective, then $\mathbf{B}\cdot f$ is the unique $L'$-structure with underlying set $A$ so that $f: \mathbf{B}\cdot f\rightarrow \mathbf{B}$ is an embedding. Also recall that we sometimes break our convention of using boldface for structures when considering points in spaces of structures.
\vspace{3 mm}

\begin{lemma}
\label{LogicAction}
Let $p\in S(G)$, and fix $m< \omega$. Let $\mathbf{A}_m'\in \mathbf{K}'(\mathbf{A}_m)$, and let $x\in X_{\mathbf{K}'}$. Then  $xp\cdot i_m = \mathbf{A}_m'$ iff $\{f\in H_m: x\cdot f = \mathbf{A}_m'\}\in p(m)$.
\end{lemma}

\begin{proof}
Let $g_i\in G$ be a net with $g_i\rightarrow p$. Then $xg_i\rightarrow xp$. Since $\mathbf{K}'(\mathbf{A}_m)$ is finite, $xg_i\cdot i_m$ eventually equals $xp\cdot i_m$. But also $xg_i\cdot i_m = x\cdot g_i|_m$. Therefore we must have $\{f\in H_m: x\cdot f = xp\cdot i_m\}\in p(m)$.
\end{proof}
\vspace{3 mm}

\begin{defin}
\label{LevelAction}
By Lemma \ref{LogicAction}, if $p\in S(G)$, $x\in X_{\mathbf{K}'}$, and $m< \omega$, then $xp\cdot i_m$ depends only on $p(m)$. If $q\in \beta H_m$, we write $xq$ for the unique expansion $\mathbf{A}_m'\in \mathbf{K}'(\mathbf{A}_m)$ with $\{f\in H_m: x\cdot f = \mathbf{A}_m'\}\in q$.
\end{defin}
\vspace{3 mm}

\begin{prop}
\label{BRLift}
Let $Y\subseteq S(G)$ be a lift of $(X_{\mathbf{K}'}, \mathbf{K}')$. Then $\phi_{\mathbf{K}'}: Y\rightarrow X_{\mathbf{K}'}$ is an isomorphism.
\end{prop}

\begin{proof}
By Lemma \ref{LogicAction}, the map $\phi_{\mathbf{K}'}^m: Y_m\rightarrow \mathbf{K}'(\mathbf{A}_m)$ sending $p$ to $\mathbf{K'}p$ is well-defined. Furthermore, since $\phi_{\mathbf{K}'}$ is surjective, $\phi_{\mathbf{K}'}^m$ must also be surjective, so $|Y_m|\geq R_m$. By Lemma \ref{CompletionLift} and Proposition \ref{BRCompletion}, we must have $|Y_m| = R_m$, and the map is bijective. 

Towards a contradiction, assume $\phi_{\mathbf{K}'}$ wasn't injective, and find $y_0\neq y_1\in Y$ with $\mathbf{K}'y_0 = \mathbf{K}'y_1$. Find $m <\omega$ with $y_0(m)\neq y_1(m)$. But then $\mathbf{K}'y_0(m)\neq \mathbf{K}'y_1(m)$, a contradiction. 
\end{proof}
\vspace{3 mm}

\begin{proof}[Proof of Theorem \ref{BRUCA}]
Instead of working with $X_{\mathbf{K}'}$, we instead use Proposition \ref{BRLift} to work with an isomorphic lift $Y\subseteq S(G)$, and let let $y_0\in Y$ be the unique point with $\mathbf{K}'y_0 = \mathbf{K}'$. Note that $y_0$ is an idempotent and a completion point. By Proposition \ref{EndoIso}, any surjective $G$-map $\psi: Y\rightarrow Y$ must be an isomorphism. Once we prove that $Y$ is a universal completion flow, it then follows that any two universal completion flows are isomorphic.
\vspace{3 mm}

Now let $W$ be another completion flow, with $w_0\in W$ a completion point. Using Lemma \ref{CompletionLift}, we may assume that $W\subseteq S(G)$ with $W_m$ finite for each $m< \omega$. For each $m< \omega$, let $\gamma_m: H_m\rightarrow Y_m$ be the coloring given by $\gamma_m(f) = y_0\cdot f$, and let $\delta_m: H_m\rightarrow W_m$ be the coloring given by $\delta_m(f) = w_0\cdot f$. For each $n < \omega$, use Proposition \ref{ColorRefine} to find $\eta_n\in H$ so that for every $m\leq n$, we have that $\gamma_m\cdot \eta_n \geq\delta_m\cdot \eta_n$. 
\vspace{3 mm}

Since $y_0$ is a completion point, we can find $p_n\in S(G)$ with $y_0\cdot \eta_n\cdot p_n = y_0$. Notice that for each $m\leq n$, we then have $\gamma_m\cdot \eta_n\cdot p_n = \gamma_m$. Let $p\in S(G)$ be a cluster point of the sequence $\eta_np_n$. Then for every $m< \omega$, we have $\gamma_m\cdot p = \gamma_m$. By Proposition \ref{RefineUnavoidable}, each $\delta_m\cdot p$ is an unavoidable $|W_m|$-coloring, and $\gamma_m\cdot p \geq \delta_m\cdot p$. Let $c_m: Y_m\rightarrow W_m$ be the surjective map so that $c_m\circ (\gamma_m \cdot p)= \delta_m\cdot p$.
\vspace{3 mm}

We now show that $w_0\cdot py_0 = w_0\cdot p$. Since $y_0$ is an idempotent and by definition of $\gamma_m$, we have $\gamma_m\cdot y_0 = \gamma_m$. Then $w_0\cdot p \cdot i_m = c_m\circ \gamma_m(i_m) = c_m\circ \gamma_m\cdot y_0(i_m) = w_0\cdot p\cdot y_0\cdot i_m$. Since this holds for each $m< \omega$, we have $w_0\cdot py_0 = w_0\cdot p$ as desired.
\vspace{3 mm}

Consider the $G$-map $\lambda_{w_0p}: Y\rightarrow W$. To check that $\lambda_{w_0p}$ is surjective, it suffices to check for each $m< \omega$ that $\lambda_{w_0p}^m: Y_m\rightarrow W_m$ is surjective. But notice that $w_0p(y_0\cdot f) = w_0p\cdot f$, and we have seen that the map $f\rightarrow w_0\cdot p\cdot f$, i.e.\ the coloring $\delta_m\cdot p$, is surjective.
\end{proof}
\vspace{3 mm}

\section{Examples of universal completion flows}
\label{ExampleSection}
\vspace{3 mm}

This section brings together some examples of automorphism groups with universal completion flows which can be computed using Theorem \ref{BRUCA}. Unlike the case with small Ramsey degrees, few examples of classes whose big Ramsey behavior has been explicitly described are known, leading to this section being relatively short. 

The simplest example of an automorphism group with a metrizable universal completion flow is the group $G = S_\infty = \aut{\mathbf{K}}$, where $\mathbf{K}$ is just a countable set with no additional structure. If we let $\mathbf{A}_n\subseteq \mathbf{K}$ be a subset of size $n$, we see by Ramsey's theorem that the small Ramsey degree and the big Ramsey degree of $\mathbf{A}_n$ are both $n!$ (recall that we are considering embedding versions of Ramsey degree, so the $n!$ comes from the automorphisms of $\mathbf{A}_n$). It follows that the universal minimal flow $M(S_\infty)$ is the universal completion flow of $S_\infty$. This is just the space of linear orders on a countable set. More generally, whenever $\mathbf{K} = \flim{\mathcal{K}}$ and $\mathcal{K}$ is a \fr class where the big and small Ramsey degrees are finite and equal, then $M(\aut{\mathbf{K}})$ is the universal completion flow of $\aut{\mathbf{K}}$. It would be interesting to find other examples of \fr classes where the big and small Ramsey degrees coincide.

\subsection{Finite distance ultrametric spaces}

Another family of examples are the classes of finite distance ultrametric spaces. Fix $S\subseteq (0, \infty)$ with $|S| = r < \omega$, and let $\mathcal{K}_0$ be the class of finite ultrametric spaces with distances from $S$. The big Ramsey behavior of these classes was described by Nguyen Van Th\'e in \cite{LionelRamsey}. To describe the big Ramsey structure, it is useful to instead work with the class $\mathcal{K}$ of rooted finite trees of height at most $r$. Structures in $\mathcal{K}$ are of the form $\langle T, \preceq, L_0,...,L_r\rangle$, where $\preceq$ is the partial order and $L_i$ is a unary predicate saying that a node is on level $i$ of the tree (it should be remarked that this class is not hereditary, but we will discuss \fr classes without HP in the next section). Then $\mathbf{K} = \flim{\mathcal{K}}$ is the rooted, countably-branching tree of height $r$. If $\mathbf{K}_0 = \flim{\mathcal{K}_0}$, then $\mathbf{K}_0$ can be identified with the set of leaves of $\mathbf{K}$, and $\aut{\mathbf{K}}\cong \aut{\mathbf{K}_0}$. Then we have the following.
\vspace{3 mm}

\begin{prop}
Let $\mathbf{K}' = \langle \mathbf{K}, \leq\rangle$, where $\leq$ is a linear order in order type $\omega$ which extends the tree order. Then $\mathbf{K}'$ is a big Ramsey structure for $\mathbf{K}$.
\end{prop}
\vspace{3 mm}

It follows that $X_{\mathbf{K}'}$, the space of linear orderings on $\mathbf{K}$ which extend the tree order, is the universal completion flow for $G = \aut{\mathbf{K}}$. It should be noted that this is not the same space as $M(G)$. Nguyen Van Th\'e describes $M(G)$ in \cite{LionelSmall}; this is the space of all \emph{convex} linear orderings on the leaves of $\mathbf{K}$. Here, a linear order of the leaves is convex if whenever $s, t, u\in \mathbf{K}$ are leaves with $s \leq t \leq u$, then the meet of $s$ and $u$ is an initial segment of $t$.

\subsection{The rational linear order}

We next consider the example from the introduction, the rational linear order $\langle \mathbb{Q}, \leq\rangle$. The group $G = \aut{\mathbb{Q}}$ is extremely amenable, but it is not hard to see that $G$ admits a non-trivial completion flow; the space of linear orders on $\mathbb{Q}$ is a good example, as for instance any linear order of order type $\omega$ is a completion point. As was mentioned in the introduction, this is not the universal completion flow. A good account of the big Ramsey behavior of $\mathbb{Q}$ can be found in Todor\v{c}evi\'c's book \cite{Tod}. 

To construct the universal completion flow, first consider the binary tree $2^{< \omega}$. If $x, y\in 2^{< \omega}$, we set $x\wedge y$ to be the longest common initial segment of both $x$ and $y$. We set $|x|$ to be the unique $n< \omega$ so that $x\in 2^n$. If $x\in 2^n$ and $m< n$, write $x|_m$ for the restriction of $x$ to domain $m$. We say that $x$ and $y$ are \emph{comparable} if either $x\wedge y = x$ or $x\wedge y = y$; otherwise we say $x$ and $y$ are \emph{incomparable}. Define $x\prec y$ if $x$ and $y$ are incomparable and $x(|x\wedge y|) < y(|x\wedge y|)$, which in the case of the binary tree means $x(|x\wedge y|) = 0$ and $y(|x\wedge y|) = 1$. A subset $A\subseteq 2^{<\omega}$ is an \emph{antichain} if no two distinct elements of $A$ are comparable. Notice that if $A$ is an antichain, then $\prec$ is a linear order on $A$. 

It is possible to build an antichain $Q\subseteq 2^{<\omega}$ so that $\langle Q, \prec \rangle \cong \langle \mathbb{Q}, <\rangle$, and we freely identify $Q$ with $\mathbb{Q}$. We now define the $4$-ary relation $R$ as follows. If $p\leq q \leq r \leq s\in \mathbb{Q}$, we set $R(p, q, r, s)$ iff $|p\wedge q| \leq |r\wedge s|$. We then have the following.
\vspace{3 mm}

\begin{prop}
The structure $\mathbf{Q}' := \langle \mathbb{Q}, \leq, R\rangle$ is a big Ramsey structure for $\langle \mathbb{Q}, \leq\rangle$.
\end{prop}
\vspace{3 mm}

We can then interpret the space $X_{\mathbf{Q}'}$ as a space of total pre-orders on $W:= \mathbb{Q}\cup [\mathbb{Q}]^2$. Let $L$ be a total pre-order of $W$, and let $E_L$ be the induced equivalence relation on $W$. Then $L\in X_{\mathbf{Q}'}$ iff $L|_{\mathbb{Q}}$ is a linear order, and given $a < b < c\in \mathbb{Q}$, we have $\neg E_L(\{a,b\}, \{b, c\})$, and $\{a, c\}$ is $E_L$-equivalent to the $L$-least of $\{a, b\}$ or $\{b, c\}$.

\subsection{The random graph}

The \emph{Random graph}, often called the \emph{Rado graph}, is the \fr limit of the class of all finite graphs. A countable graph $\langle Q, E\rangle$ is isomorphic to the Rado graph iff for any disjoint and finite $F_0, F_1\subseteq Q$, then there is $x\in Q\setminus (F_0\cup F_1)$ so that $\neg E(x, y)$ for each $y\in F_0$ and $E(x, z)$ for every $z\in F_1$. 

It can be shown that the big Ramsey degree of any  finite subgraph of the Rado graph is finite by using Milliken's tree theorem. To construct a big Ramsey structure, we follow the presentation of Laflamme, Sauer, and Vuksanovic \cite{LSV}. Once again, we consider the binary tree $2^{<\omega}$. We call a subset $T\subseteq 2^{<\omega}$ \emph{transversal} if $|x|\neq |y|$ for any distinct $x, y\in T$. If $T\subseteq 2^{<\omega}$ is transversal, we can give $T$ a graph structure $E$, where if $x, y\in T$ and $|x| < |y|$, we set $E(x, y)$ iff $y(|x|) = 1$. Now let $\langle Q, E\rangle$ be a Rado graph, and fix an enumeration $Q = \{q_n: n< \omega\}$. To each $q_n$, we associate an element $x_n\in 2^n$, where for $m < n< \omega$, we set $x_n(m) = 1$ iff $E(q_m, q_n)$. 

Theorem 7.6 from \cite{LSV} now gives us an unavoidable coloring for each finite subgraph. To turn this into a Ramsey structure, we need to perform one extra step. Find a subset $T\subseteq Q$ of the Rado graph so that $\langle T, E|_T\rangle$ is isomorphic to the Rado graph, and so that $\{x_n: q_n\in T\}\subseteq 2^{<\omega}$ is an antichain. By doing this, we can ensure that the collection of ``non-diagonal'' tuples as defined in \cite{LSV} is empty. 

To describe the resulting structure, it will be useful to instead assume that we have mapped $Q$ into a transversal antichain in $2^{<\omega}$ which respects the graph structure. With this identification, we now define $\prec$ and the $4$-ary relation $R$ as before.

\begin{prop}
The structure $\mathbf{Q}':= \langle Q, E, \prec, R\rangle$ is a big Ramsey structure for the Rado graph $\langle Q, E\rangle$.
\end{prop}

Similar to the example of the rationals, the space $X_{\mathbf{Q}'}$ can be described as a space of pairs $(L_0, L_1)$, where $L_0$ is a linear order of $Q$ and $L_1$ is a total preorder of $Q\cup [Q]^2$. Describing precisely which pairs are in the closure of the big Ramsey structure seems to be somewhat more difficult. 

\subsection{The orders $\mathbb{Q}_n$ and the tournament $\mathbf{S}(2)$}

The last examples we will consider are the dense local order $\mathbf{S}(2)$ and the orders $\mathbb{Q}_n$.  The dense local order is a countable \emph{tournament}, a directed graph $\langle S, E\rangle$ where for distinct $x, y\in S$ exactly one of $E(x,y)$ or $E(y, x)$ holds. One way to construct $\mathbf{S}(2)$ is to consider a countable dense set of points on the unit circle so that no two points are exactly $\pi$ radians apart. Then set $E(x, y)$ iff $y$ is less than $\pi$ radians counterclockwise from $x$. Then $\langle S, E\rangle$ is isomorphic to $\mathbf{S}(2)$.

The big Ramsey behavior of the structure $\mathbf{S}(2)$ is studied by Laflamme, Nguyen Van Th\'e, and Sauer in \cite{LNS}. The trick to analyzing $\mathbf{S}(2)$ is to instead analyze the structure $\mathbb{Q}_2 := \langle \mathbb{Q}, <, P_0, P_1\rangle$, where $\langle \mathbb{Q}, <\rangle$ is the rational order, and each $P_i$ is a dense subset of $\mathbb{Q}$ with $\mathbb{Q} = P_0\sqcup P_1$. The structures $\mathbb{Q}_n$ are defined similarly; they are rational orders with a distinguished partition into $n$ dense pieces. The authors of \cite{LNS} prove a slight generalization of Milliken's theorem to obtain big Ramsey results for the structures $\mathbb{Q}_n$, the ``colored'' version alluded to in the title of \cite{LNS}. However, once this is proven, the big Ramsey structures for $\mathbb{Q}_n$ are easy to describe; namely, if $\langle \mathbb{Q}, <, R\rangle$ is a big Ramsey structure for the rational order, then $\langle \mathbb{Q}, <, P_0,...P_{n-1}, R\rangle$ is a big Ramsey structure for $\mathbb{Q}_n$.

Using the big Ramsey structure for $\mathbb{Q}_2$, one obtains a big Ramsey structure for $\mathbf{S}(2)$ as follows. Represent $\mathbf{S}(2)$ as $\langle S, E\rangle$, where $S$ is a dense subset of the unit circle as before. Then we can view $\mathbb{Q}_2$ as a structure with underlying set $S$. We let $P_0$ be those points below the $x$-axis, and $P_1$ be the points above. Let $S^1$ be the unit circle; define the map $\phi: S\rightarrow S^1$ by setting $\phi(x) = x$ for $x\in P_0$ and $\phi(x) = x\cdot e^{i\pi}$ for $x\in P_1$. Note that $\phi$ is an injection with $\im{\phi}$ contained below the $x$-axis. Then for $x, y\in S$, we set $x < y$ iff $\phi(y)$ is to the right of $\phi(x)$. Then if $\mathbb{Q}_2' :=\langle S, <, P_0, P_1, R\rangle$ is a big Ramsey structure for $\mathbb{Q}_2$, then $\mathbf{S}' := \langle S, E, <, P_0, P_1, R\rangle$ is a big Ramsey structure for $\mathbf{S}(2)$.
\vspace{3 mm}

\section{Can we find big Ramsey structures?}
\label{ExpansionSection}
\vspace{3 mm}

Theorem \ref{BRUCA} makes it important to know when a \fr structure $\mathbf{K} = \flim{\mathcal{K}}$ admits a big Ramsey structure. An obvious necessary condition is that every $\mathbf{A}\in \mathcal{K}$ have finite big Ramsey degree. However, this seems far from sufficient. Suppose $\mathbf{K}'$ is a big Ramsey structure and $\mathbf{K} = \bigcup_n \mathbf{A}_n$, where each $\mathbf{A}_n$ has big Ramsey degree $R_n < \omega$. For each $n< \omega$, consider the coloring $\gamma_n: H_n\rightarrow \mathbf{K}'(\mathbf{A}_n)$, where for $s\in H_n$, we set $\gamma_n(s) = \mathbf{K}'\cdot s$. Then each $\gamma_n$ is an unavoidable $R_n$-coloring, and furthermore $\gamma_m \ll \gamma_n$ whenever $m \leq n$. As hinted in the discussion near the end of section \ref{SetsAndColorings}, this is actually sufficient.
\vspace{3 mm} 

\begin{theorem}
\label{LevelsBR}
Let $\mathbf{K} = \bigcup_n \mathbf{A}_n$ be a \fr structure, and suppose each $\mathbf{A}_m$ has finite big Ramsey degree $R_m < \omega$. Assume that for each $m <\omega$, there is an unavoidable $R_m$-coloring $\gamma_m$ of $H_m$ so that $\gamma_m \ll \gamma_n$ for each $m\leq n < \omega$. Then $\mathbf{K}$ admits a big Ramsey structure.
\end{theorem}
\vspace{3 mm}

While fairly intuitive, the proof of Theorem \ref{LevelsBR} is surprisingly involved and will be relegated to the appendix. One way of interpreting Theorem \ref{LevelsBR} is that it justifies our approach of always fixing an exhaustion $\mathbf{K} = \bigcup_n \mathbf{A}_n$ and only paying attention to the $\mathbf{A}_n$. This ``non-hereditary'' approach can be formalized using the notion of a \emph{Fra\"iss\'e--HP} class, that is a class of finite structures satisfying every property of being a \fr class except perhaps the hereditary property. If $\mathcal{K}$ is a Fra\"iss\'e--HP class,  we can still form the \fr limit $\mathbf{K} = \flim{\mathcal{K}}$. This structure has the property that for any $\mathbf{A}\subseteq \mathbf{K}$ with $\mathbf{A}\in \mathcal{K}$ and embedding $f: \mathbf{A}\rightarrow \mathbf{K}$, there is $g\in \aut{\mathbf{K}}$ with $g|_\mathbf{A} = f$. So if $\mathbf{K} = \bigcup_n{\mathbf{A}_n}$ is a \fr structure, then the class $\mathcal{K} := \{\mathbf{A}: \exists n< \omega (\mathbf{A}\cong \mathbf{A}_n)\}$ is a Fra\"iss\'e--HP class. In fact, we will use this so frequently that we now \textbf{adopt it as a notational convention}: whenever $\mathbf{K} = \bigcup_n \mathbf{A}_n$ is a \fr structure with a fixed exhaustion and we write $\mathbf{K} = \flim{\mathcal{K}}$, we intend that $\mathcal{K} := \{\mathbf{A}: \exists n< \omega (\mathbf{A}\cong \mathbf{A}_n)\}$.

For the rest of the section, fix a \fr $L$-structure $\mathbf{K} = \bigcup_n \mathbf{A}_n = \flim{\mathcal{K}}$, with $G = \aut{\mathbf{K}}$ and $\widehat{G} = \emb{\mathbf{K}}$, and assume that each $\mathbf{A}_n$ has big Ramsey degree $R_n < \omega$. The theme of this section is that even if we don't know that $\mathbf{K}$ admits a big Ramsey structure, we can say many things about what such a structure must look like. We will show that $\mathcal{K}$ admits a unique \emph{big Ramsey expansion class} $\mathcal{K}'$; if $\mathbf{K}'$ is a big Ramsey structure for $\mathbf{K}$, then $\mathcal{K}'$ will be in a suitable sense isomorphic to $\age{\mathbf{K}'}$. 

If we add some extra assumptions to $\mathbf{K}$, we can say much more about the big Ramsey expansion $\mathcal{K}'$ of $\mathcal{K}$. If we know that $G$ is \emph{Roelcke precompact}, then we will show that $\mathcal{K}'$ is itself a Fra\"iss\'e--HP class. This is simultaneously fascinating structural information and a frustrating obstacle to actually constructing a big Ramsey structure; we will discuss this at the end of the section. 

\subsection{The big Ramsey expansion class}
An \emph{expansion} of $\mathcal{K}$ is a class $\mathcal{K}'$ of $L'$-structures for some language $L'\supseteq L$ satisfying the following requirements.

\begin{enumerate}
\item
If $\mathbf{A}'\in \mathcal{K}'$, then $\mathbf{A}'|_L \in \mathcal{K}$.

\item
Every $\mathbf{A}\in \mathcal{K}$ admits an \emph{expansion}, a structure $\mathbf{A}'\in \mathcal{K}'$ with $\mathbf{A}'|_L = \mathbf{A}$. We write $\mathcal{K}'(\mathbf{A}) = \{\mathbf{A}'\in \mathcal{K}': \mathbf{A}'|_L = \mathbf{A}\}$.

\item
Suppose $\mathbf{A}\in \mathcal{K}$ and $\mathbf{B}'\in \mathcal{K}'$ with $\mathbf{B} = \mathbf{B}'|_L$. Then if $f: \mathbf{A}\rightarrow \mathbf{B}$ is an embedding, we have $\mathbf{B}\cdot f\in \mathcal{K}'(\mathbf{A})$.
\end{enumerate}

An expansion $\mathcal{K}'$ of $\mathcal{K}$ is called \emph{precompact} if $\mathcal{K}'(\mathbf{A})$ is finite for each $\mathbf{A}\in \mathcal{K}$. The expansion $\mathcal{K}'$ is called \emph{reasonable} if for every embedding $f: \mathbf{A}\rightarrow \mathbf{B}$ with $\mathbf{A}, \mathbf{B}\in \mathcal{K}$, the dual map $\hatf: \mathcal{K}'(\mathbf{B})\rightarrow \mathcal{K}'(\mathbf{A})$ given by $\hatf(\mathbf{B}') = \mathbf{B}'\cdot f$ is surjective. We are overloading the ``dual map'' notation, but the context should typically be clear. It is worth pointing out that to check if an expansion $\mathcal{K}'$ is precompact, it suffices to check that $\mathcal{K}'(\mathbf{A}_n)$ is finite for each $n< \omega$. Similarly, to check if $\mathcal{K}'$ is reasonable, it suffices to check the surjectivity of each $\hatf$ for each $f\in H_m^n$ and $m\leq n< \omega$.

If $\mathcal{K}'$ and $\mathcal{K}^*$ are expansions of $\mathcal{K}$ in languages $L'$ and $L^*$, respectively, then a \emph{map of expansions} from $\mathcal{K}'$ to $\mathcal{K}^*$ is a map $\Phi: \bigcup_n \mathcal{K}'(\mathbf{A}_n)\rightarrow \bigcup_n \mathcal{K}^*(\mathbf{A}_n)$ satisfying the following.

\begin{enumerate}
\item
If $\mathbf{A}_n'\in \mathcal{K}'$ is an expansion of $\mathbf{A}_n$, then $\Phi(\mathbf{A}_n')\in \mathcal{K}^*$ is also an expansion of $\mathbf{A}_n$.

\item
If $m\leq n <\omega$ and $f\in H_m^n$, then $\Phi(\mathbf{A}_n'\cdot f) = \Phi(\mathbf{A}_n')\cdot f$.
\end{enumerate} 

We call $\mathcal{K}'$ and $\mathcal{K}^*$ \emph{isomorphic expansions} if there is a bijective (equivalently invertible) map of expansions from $\mathcal{K}'$ to $\mathcal{K}^*$. Much more on Fra\"iss\'e--HP classes and their expansions can be found in section 5 of \cite{Z}. One caution to the interested reader: the definition of expansion given there is missing the analog of item (3) from the definition here. This property of expansions is used implicitly throughout \cite{Z} and needs to be included in the definition. I thank Aleksandra Kwiatkowska for pointing this out to me.

We can now begin working towards the definition of the big Ramsey expansion of $\mathcal{K}$.
\vspace{3 mm}

\begin{defin}\mbox{}
\label{DiagramsDef}
\begin{enumerate}
\item
Suppose $m\leq n < \omega$. An \emph{$(m,n)$-diagram} is any map $D: J_n\times H_m^n\rightarrow J_m$ such that $J_n$ and $J_m$ are finite and so that for every $f\in H_m^n$, the map $D(-, f): J_n\rightarrow J_m$ is surjective.

\item
Let $D_J: J_n\times H_m^n\rightarrow J_m$ and $D_I: I_n\times H_m^n\rightarrow I_m$ be $(m,n)$-diagrams. An \emph{isomorphism of $(m,n)$-diagrams}, written $\sigma: D_J\Rightarrow D_I$, is a pair $\sigma := (\sigma_m, \sigma_n)$ of bijections  $\sigma_m: J_m\rightarrow I_m$ and $\sigma_n: J_n\rightarrow I_n$ so that the following diagram commutes.

\begin{center}
\begin{tikzpicture}[node distance=4cm, auto]
\node (A) {$J_n\times H_m^n$};
\node (B) [right of=A] {$I_n\times H_m^n$};
\node (C) [node distance=2cm, below of=A] {$J_m$};
\node (D) [right of=C] {$I_m$};
\draw[->] (A) to node {$\sigma_n\times 1$} (B);
\draw[->] (A) to node [swap] {$D_J$} (C);
\draw[->] (B) to node {$D_I$} (D);
\draw[->] (C) to node {$\sigma_m$} (D);
\end{tikzpicture}
\end{center}

\item
Let $r \leq \omega$. An \emph{$r$-diagram based on $\{J_n: n< r\}$} is a collection $D= \{D(m,n): m\leq n < r\}$ satisfying the following properties.

\begin{enumerate}
\item
Each $J_n$ is a finite set so that for every $m\leq n< r$, $D(m,n): J_n\times H_m^n\rightarrow J_m$ is an $(m,n)$-diagram. Furthermore, $|J_0| = 1$.

\item
If $m\leq n\leq N < r$, $f\in H_m^n$, $s\in H_n^N$, and $j\in J_N$, then $D(m,N)(j, s\circ f) = D(m,n)(D(n,N)(j, s), f)$.
\end{enumerate}

\item
Let $D_J = \{D_J(m,n): m\leq n< r\}$ and $D_I = \{D_I(m,n): m\leq n< r\}$ be $r$-diagrams based on $\{J_n: n< r\}$ and $\{I_n: n< r\}$, respectively. An \emph{isomorphism of $r$-diagrams} $\sigma: D_J\Rightarrow D_I$ is a tuple $\sigma := \{\sigma_n: n< r\}$ so that for every $m\leq n < r$, $(\sigma_m, \sigma_n): D_J(m,n)\Rightarrow D_I(m,n)$ is an isomorphism. 

\item
If $r\leq \omega$, $N\leq r$, and $D := \{D(m,n): m\leq n< r\}$ is an $r$-diagram, then the \emph{restriction} of $D$ to $N$ is the $N$-diagram $D|_N:= \{D(m,n): m\leq n< N\}$.
\end{enumerate}
\end{defin}
\vspace{3 mm}

\begin{exa}\mbox{}
\label{DiagramsExa}
\begin{enumerate}
\item
Suppose $r< \omega$, and let $\gamma_0 \ll \gamma_1 \ll \cdots \ll \gamma_{r-1}$ be colorings of $H_0$,...,$H_{r-1}$, respectively. For $m\leq n < r$, the \emph{diagram} of $\gamma_m \ll \gamma_n$ is the map $D(\gamma_m, \gamma_n): \im{\gamma_n}\times H_m^n \rightarrow \im{\gamma_m}$ where given $f\in H_m^n$ and $j\in \im{\gamma_n}$, we set $D(\gamma_m, \gamma_n)(j, f) = i$ if for any $s\in H_n$ with $\gamma_n(s) = j$, we have $\gamma_m(s\circ f) = i$. The \emph{diagram} of $\gamma:= \{\gamma_0,...,\gamma_{r-1}\}$ is the collection $D_\gamma := \{D(\gamma_m, \gamma_n): m\leq n< r\}$

\item
If $\mathcal{K}'$ is a reasonable, precompact expansion of $\mathcal{K}$ and $m\leq n< \omega$, then the diagram $D_{\mathcal{K}'}(m, n): \mathcal{K}'(\mathbf{A}_n)\times H_m^n\rightarrow \mathcal{K}'(\mathbf{A}_m)$ is given by $D_{\mathcal{K}'}(m, n)(\mathbf{A}_n', f) = \mathbf{A}_n'\cdot f$. We form an $\omega$-diagram by setting $D_\mathcal{K}' = \{D_\mathcal{K}'(m,n): m\leq n < \omega\}$. 

\item
Suppose $\mathcal{K}'$ and $\mathcal{K}^*$ are two reasonable, precompact expansions of $\mathcal{K}$. Then $\mathcal{K}'$ and $\mathcal{K}^*$ are isomorphic expansions iff $D_{\mathcal{K}'}$ and $D_{\mathcal{K}^*}$ are isomorphic diagrams.
\end{enumerate}
\end{exa}
\vspace{3 mm}

\begin{prop}
\label{EqualDiagrams}
Suppose $r <\omega$. Let $\gamma := \{\gamma_0\ll\cdots \ll \gamma_{r-1}\}$ and $\delta:= \{\delta_0 \ll\cdots \ll \delta_{r-1}\}$, where for each $n< r$, we have $\gamma_n$ and $\delta_n$ unavoidable $R_n$ colorings. Then $D_\gamma$ and $D_\delta$ are isomorphic $r$-diagrams.
\end{prop}

\begin{proof}
Notice first that if $\eta\in \widehat{G}$ and $m\leq n< r$, then $D(\gamma_m, \gamma_n) = D(\gamma_m\cdot \eta, \gamma_n\cdot \eta)$, and similarly for $\delta_m$ and $\delta_n$. Use Proposition \ref{ColorRefine} several times to find $\eta\in \widehat{G}$ with $\gamma_n\cdot \eta \sim \delta_n\cdot \eta$ for every $n< r$. For $n< r$, let $\sigma_n: R_n\rightarrow R_n$ be the bijections so that $\sigma_n\circ \gamma_n\cdot \eta = \delta_n\cdot \eta$, and set $\sigma := \{\sigma_n: n< r\}$. Now for any $m\leq n< r$, we have the following commutative diagram, showing that $\sigma: D_\gamma\Rightarrow D_\delta$ is an isomorphism.
\[
\begin{tikzpicture}[node distance=4cm, auto]
\node (A) {$R_n\times H_m^n$};
\node (B) [right of=A] {$R_n\times H_m^n$};
\node (C) [node distance=2cm, below of=A] {$R_m$};
\node (D) [right of=C] {$R_m$};
\draw[->] (A) to node {$\sigma_n\times 1$} (B);
\draw[->] (A) to node [swap] {$D(\gamma_m, \gamma_n)$} (C);
\draw[->] (B) to node {$D(\delta_m, \delta_n)$} (D);
\draw[->] (C) to node {$\sigma_m$} (D);
\end{tikzpicture}
\qedhere
\]
\qedhere
\end{proof}
\vspace{0 mm}

\begin{rem}
If $r< \omega$, we write $D_r$ for any $r$-diagram based on $\{R_n: n< r\}$ isomorphic to $D_\gamma$ as in Proposition \ref{EqualDiagrams}.
\end{rem}
\vspace{3 mm}

\begin{prop}
\label{OmegaDiagrams}
There is up to isomorphism a unique $\omega$-diagram $D_\omega$ with $D_\omega|_r \cong D_r$ for every $r< \omega$. 
\end{prop}

\begin{proof}
Consider the tree $T$ where level $r$ is the collection of $r$-diagrams based on $\{R_n: n< r\}$ isomorphic to $D_r$. This is an infinite finitely branching tree, so by K\"onig's lemma, T has an infinite branch $B$. Set $D_\omega := \bigcup_n B(n)$. 

Suppose $D$ is another $\omega$-diagram with $D|_r \cong D_r$ for every $r< \omega$. Consider the tree $T'$ where level $r$ is the collection of isomorphisms $(\sigma_0,...,\sigma_{r-1}): D_\omega|_r\Rightarrow D|_r$. This is also an infinite finitely branching tree, so let $B'\subseteq T'$ be an infinite branch. Then setting $\sigma:= \bigcup_n B'(n)$, then $\sigma: D_\omega\Rightarrow D$ is the desired isomorphism.
\end{proof}
\vspace{3 mm}

\begin{rem}
We call $D_\omega$ the \emph{big Ramsey diagram} of $\mathcal{K}$. We usually take $D_\omega$ to be based on $\{R_n: n< \omega\}$ unless otherwise specified. 
\end{rem}
\vspace{3 mm}

\begin{prop}
\label{DiagramExpansion}
Let $D := \{D(m,n): m\leq n< \omega\}$ be an $\omega$-diagram. Then there is up to isomorphism a unique reasonable, precompact expansion $\mathcal{K}'$ of $\mathcal{K}$ so that $D\cong D_{\mathcal{K}'}$.
\end{prop}

\begin{proof}
Suppose $D$ is based on $\{J_n: n< \omega\}$. For each $n< \omega$, set $r_n = |\mathbf{A}_n|$, and fix an enumeration $K := \{a_i: i< \omega\}$ so that $A_n = \{a_0,...,a_{r_n-1}\}$. We form a language $L'\supseteq L$ by introducing for each $q\in J_n$ an $r_n$-ary relation symbol $S_n^q$. If $n< \omega$ and $q\in J_n$, we define the expansion $\mathbf{A}_n^q$ of $\mathbf{A}_n$ as follows. Let $m\leq n$, and let $b_0,...,b_{r_m-1}\in A_n$ be an $r_m$-tuple. If $p\in J_m$, then $(S_m^p)^{\mathbf{A}_n^q}(b_0,...,b_{r_m-1})$ holds iff the map $f: A_m\rightarrow A_n$ with $f(a_i) = b_i$ is in $H_m^n$ and if $D(m,n)(q, f) = p$. We then set $\mathcal{K}' = \{\mathbf{A}': \exists n< \omega \exists q\in J_n (\mathbf{A}'\cong \mathbf{A}_n^q)\}$. 

The key observation about this definition is as follows. Suppose $\mathbf{A}'\in \mathcal{K}$, and let $b_0,...,b_{r_m-1}\in A'$ be an $r_m$-tuple. If $p\in J_m$ and $(S_m^p)^{\mathbf{A}'}(b_0,...,b_{r_m-1})$ holds, then the map $f: A_m\rightarrow A'$ with $f(a_i) = b_i$ is an embedding of $\mathbf{A}_m$ into $\mathbf{A}'$. Furthermore, we have $\mathbf{A}'\cdot f = \mathbf{A}_m^p$; this is a simple consequence of item (3b) of Definition \ref{DiagramsDef}. With this observation, checking that $\mathcal{K}'$ works is simple. Since each $J_n$ is finite, $\mathcal{K}'$ is precompact. Since for every $f\in H_m^n$ the map $D(m,n)(-, f)$ is surjective, $\mathcal{K}'$ is reasonable. Define $\sigma := \{\sigma_n: n< \omega\}$, where $\sigma_n: J_n\rightarrow \mathcal{K}'(\mathbf{A}_n)$ and $\sigma_n(q) = \mathbf{A}_n^q$; then $\sigma: D\Rightarrow D_{\mathcal{K}'}$ is an isomorphism. 

Now suppose $\mathcal{K}^*$ is another expansion of $\mathcal{K}$ with $D\cong D_{\mathcal{K}^*}$. Then $D_{\mathcal{K}^*}\cong D_{\mathcal{K}'}$, so $\mathcal{K}$ and $\mathcal{K}'$ are isomorphic expansions by item (3) from Example \ref{DiagramsExa}
\end{proof}
\vspace{3 mm}

\begin{defin}
\label{BigRamseyExp}
We call any reasonable, precompact expansion $\mathcal{K}'$ of $\mathcal{K}$ with $D_{\mathcal{K}'}\cong D_\omega$ the \emph{big Ramsey expansion} of $\mathcal{K}$.
\end{defin}
\vspace{3 mm}

We end this subsection by showing that if $\mathbf{K}$ admits a big Ramsey structure $\mathbf{K}'$, the big Ramsey expansion $\mathcal{K}'$ of $\mathcal{K}$ is more-or-less the same thing as $\age{\mathbf{K}'}$. We need to briefly discuss what ``age'' means in the non-hereditary context. Suppose $L^*\supseteq L$ is a language and $\mathbf{K}^*$ is an $L^*$-structure with $\mathbf{K}^*|_L = \mathbf{K}$. Then the \emph{age of $\mathbf{K}^*$ over $\mathcal{K}$} is the class $\age{\mathbf{K}^*/\mathcal{K}} := \{\mathbf{A}^*: \exists n< \omega \exists s\in H_n(\mathbf{A}^*\cong \mathbf{K}^*\cdot s)\}$. This is a reasonable expansion of $\mathcal{K}$.
\vspace{3 mm}

\begin{theorem}
Suppose that $\mathbf{K}$ admits a big Ramsey structure $\mathbf{K}'$, and let $\mathcal{K}'$ be the big Ramsey expansion of $\mathcal{K}$. Then $\mathcal{K}'$ and $\age{\mathbf{K}'/\mathcal{K}}$ are isomorphic expansions.
\end{theorem}

\begin{proof}
Write $\mathcal{K}^*$ for $\age{\mathbf{K}'/\mathcal{K}}$. First notice that $|\mathcal{K}^*(\mathbf{A}_n)| = R_n$ for every $n< \omega$, so $\mathcal{K}^*$ is precompact. As $\mathcal{K}^*$ is also reasonable, we can now form the diagram $D_{\mathcal{K}^*}$, and it suffices to show that $D_{\mathcal{K}'}\cong D_{\mathcal{K}^*}$. For each $n < \omega$, consider the coloring $\gamma_n$ of $H_n$ given by $\gamma_n(s) = \mathbf{K}'\cdot s$. Each $\gamma_n$ is an unavoidable $R_n$-coloring, and $\gamma_m\ll \gamma_n$ whenever $m\leq n < \omega$, so we can use these colorings to produce the diagram $D_\omega\cong D_{\mathcal{K}'}$. Using this representation of $D_\omega$, if $m\leq n< \omega$, $s\in H_n$, and $f\in H_m^n$, we can write $D_\omega(m,n)(\mathbf{K}'\cdot s, f) = \mathbf{K}'\cdot s\circ f$. But also $D_{\mathcal{K}^*}(m,n)(\mathbf{K}'\cdot s, f) = \mathbf{K}'\cdot s\circ f$. So $D_{\mathcal{K}'}\cong D_{\mathcal{K}^*}$.
\end{proof}
\vspace{3 mm}

\subsection{Roelcke precompact automorphism groups}
\label{RoelckeSection}
\vspace{3 mm}

\begin{defin}
A topological group $\Gamma$ is said to be \emph{Roelcke precompact} if for any open $U\subseteq G$ with $1_G\in U$, there is a finite $F\subseteq G$ with $UFU = \Gamma$. 
\end{defin}
\vspace{3 mm}

The class of Roelcke precompact automorphism groups is quite robust. For instance, if $\mathcal{L}$ is a \fr class (\emph{with} HP) in a finite relational language, then $\aut{\flim{\mathcal{L}}}$ is Roelcke precompact. This subsection first provides some background on what it means for $G = \aut{\mathbf{K}}$ to be Roelcke precompact, in particular, what it means for the class $\mathcal{K}$. We then show that if $G$ is Roelcke precompact, then the big Ramsey expansion class $\mathcal{K}'$ is in fact a Fra\"iss\'e--HP class.

For the rest of this section, fix $G = \aut{\mathbf{K}}$ for some \fr structure $\mathbf{K} = \flim{\mathcal{K}} = \bigcup_n \mathbf{A}_n$. Let $\widehat{G} = \emb{\mathbf{K}}$ be the left completion.
\vspace{3 mm}

\begin{lemma}
\label{Roelcke}
$G$ is Roelcke precompact iff for evey $m < \omega$, there is $n \geq m$ so that for any $f_0, f_1\in H_m$, there is $s\in H_n$ with $f_0 = s\circ i_m$ and $f_1\in \{s\circ f: f\in H_m^n\}$.
\end{lemma}

\begin{proof}
Suppose $G$ is Roelcke precompact, and fix $m < \omega$. Then $U = \{g\in G: g|_m = i_m\}$ is an open neighborhood of the identity. Find $F:= \{g_0,...,g_{k-1}\}\subseteq G$ with $UFU = G$. Let $n\geq m$ be large enough so that for each $i< k$ we have $g_i|_m\in \emb{\mathbf{A}_m, \mathbf{A}_n}$.

Let $f_0, f_1\in H_m$. By ultrahomogeneity, we may assume that $f_0 = i_m$. Find $i < k$ and $u\in U$ with $ug_i\circ f_0 = f_1$. It follows that $(u\circ i_n)\circ i^n_m = i_m$ and $(u\circ i_n)\circ g_i|_m = f_1$ as desired.

Conversely, assume that $G$ has the property stated in the lemma. Let $U\subseteq G$ be an open neighborhood of $1_G$. We may assume $U = \{g\in G: g|_m = i_m\}$ for some $m< \omega$. Let $n\geq m$ witness the property from the lemma, and find $F:= \{g_0,...,g_{k-1}\}\subseteq G$ so that $\{g_i|_m: i< k\} = \emb{\mathbf{A}_m, \mathbf{A}_n}$. 

Let $g\in G$. Find $i< k$ and $s\in H_n$ so that $s\circ i_m = i_m$ and $s\circ g_i|_m = g_m$. Find $u\in G$ with $u|_n = s$, and notice that $u\in U$. Since $g_i^{-1}u^{-1}g\in U$, we have $g\in UFU$ as desired. 
\end{proof}
\vspace{3 mm}

Another useful way to think about Roelcke precompactness is via the notion of a ``type.'' Though the definition we present might look different, this is the same notion of type as from model theory.
\vspace{3 mm}

\begin{defin}
Let $m< \omega$, and let $(f_0,...,f_{k-1})$ and $(h_0,...,h_{k-1})$ be $k$-tuples from $H_m$. We say that $(f_0,...,f_{k-1})$ and $(h_0,...,h_{k-1})$ \emph{have the same type} if there is $g\in G$ with $g\circ f_i = h_i$ for each $i< k$. A \emph{$k$-type} on $H_m$ is any equivalence class of $k$-tuples. We write $\tp{f_0,...,f_{k-1}}$ for the type that $(f_0,...,f_{k-1})$ belongs to. Write $S_m^{(k)}$ for the collection of $k$-types over $H_m$.
\end{defin}
\vspace{3 mm}

\begin{rem}
For $(f_0,...,f_{k-1})$ and $(h_0,...,h_{k-1})$ to have the same type, it is sufficient to find $\eta\in \widehat{G}$ with $\eta\circ f_i = h_i$ for each $i< k$.
\end{rem}
\vspace{3 mm}

\begin{lemma}
\label{Types}
$G$ is Roelcke precompact iff for every $m< \omega$, there are only finitely many $2$-types on $H_m$. 
\end{lemma}

\begin{proof}
Assume that there are only finitely many $2$-types. Using ultrahomogeneity, we can find $h_0,...,h_{k-1}$ so that $\{\tp{i_m, h_i}: i< k\}$ is the set of $2$-types on $H_m$. Find $n\geq m$ large enough so that $h_i\in \emb{\mathbf{A}_m, \mathbf{A}_n}$ for each $i< k$. Let $f_0, f_1\in H_m$. Find $i< k$ so that $\tp{f_0, f_1} = \tp{i_m, h_i}$, and find $g\in G$ with $g\circ i_m = f_0$ and $g\circ h_i = f_1$. Then $g|_n\circ i_m = f_0$ and $g|_n\circ h_i = f_1$, so $G$ is Roelcke precompact by Lemma \ref{Roelcke}.

Conversely, if $G$ is Roelcke precompact, use Lemma \ref{Roelcke} to find $n\geq m$ as guaranteed by the lemma. Then if $f_0, f_1\in H_m$, we have $\tp{f_0, f_1} = \tp{i_m, h}$ for some $h\in \emb{\mathbf{A}_m, \mathbf{A}_n}$, so there are only finitely many $2$-types on $H_m$.
\end{proof}
\vspace{3 mm}

We now fix $\mathcal{K}'$ the big Ramsey expansion class of $\mathcal{K}$ with the goal of showing that $\mathcal{K}'$ is a Fra\"iss\'e--HP class whenever $G$ is Roelcke precompact. It remains to show that $\mathcal{K}'$ has the Joint Embedding Property (JEP) and the Amalgamation Property (AP). Both of these properties are defined in the introduction immediately before Definition \ref{BigRamsey}. It will be useful to rephrase both of these in terms of the diagram $D_{\mathcal{K}'} \cong D_\omega$. We start with the JEP. Let $D$ be an $\omega$-diagram based on $\{J_n: n <\omega\}$. We say that $D$ has the \emph{JEP for diagrams} if for any $m< \omega$ and $p,q\in J_m$, there are $n\geq m$, $q'\in J_n$, and $f\in H_m^n$ so that $D(m,n)(q', i_m) = p$ and $D(m,n)(q', f) = q$. It is routine to check that $D_{\mathcal{K}'}$ has the JEP for diagrams iff $\mathcal{K}'$ has the JEP.
\vspace{3 mm}

\begin{theorem}
\label{RamseyJEP}
Assume that $G$ is Roelcke precompact, and let $\mathcal{K}'$ be the big Ramsey expansion class of $\mathcal{K}$. Then $\mathcal{K}'$ has the JEP.
\end{theorem}

\begin{proof}
Fix a representation of $D_\omega$ based on $\{R_n: n< \omega\}$. Fix $m< \omega$, and let $p, q\in R_m$. Using Roelcke precompactness, find $n\geq m$ large enough so that $\{\mathrm{tp}(i_m, f): f\in H_m^n\} = S_m^{(2)}$. Fix unavoidable colorings $\gamma:= \{\gamma_0\ll \cdots \ll \gamma_n\}$ so that $D_\gamma = D_\omega|_{n+1}$. For each $f\in H_m^n$, set $T_f = \{s\in H_n: \gamma_m(s\circ i_m) = p \text{ and }\gamma_m(s\circ f) = q\}$. If for some $f\in H_m^n$ we have $T_f\neq \emptyset$, then we will be done by picking $s\in T_f$ and setting $q' = \gamma_n(s)$. To see that some $T_f$ is non-empty, pick $h_p, h_q\in H_m$ with $\gamma_m(h_p) = p$ and $\gamma_m(h_q)$. By choice of $n$, we can find $s\in H_n$ and $f\in H_m^n$ with $s\circ i_m = h_p$ and $s\circ f = h_q$. Then $s\in T_f$ as desired.
\end{proof}
\vspace{3 mm}

We now turn towards the AP. If $D$ is an $\omega$-diagram based on $\{J_n: n< \omega\}$, then we say that $D$ has the \emph{AP for diagrams} if for any $m\leq n < \omega$, any $p, q\in J_n$, and any $f_p, f_q\in H_m^n$ with $D(m,n)(p, f_p) = D(m,n)(q, f_q) = v$ for some $v\in J_m$, then there are $N\geq n$, $q'\in J_N$, and $s_p, s_q\in H_n^N$ with $D(n,N)(q', s_p) = p$ and $D(n,N)(q', s_q) = q$. Once again, it is routine to check that $D_{\mathcal{K}'}$ has the AP for diagrams iff $\mathcal{K}'$ has the AP.

Our strategy for proving that $\mathcal{K}'$ has the AP is adapted from the proof of a theorem of Ne\v{s}et\v{r}il and R\"odl \cite{NeRo}. If $\mathcal{L}$ is a class of finite structures, we say that $\mathcal{L}$ has the Ramsey Property (RP) if for any $\mathbf{A}\leq \mathbf{B}\in \mathcal{L}$, there is $\mathbf{C}\in \mathcal{L}$ with $\mathbf{B}\leq \mathbf{C}$ so that $\mathbf{C}\rightarrow (\mathbf{B})^\mathbf{A}_{2}$. The theorem of Ne\v{s}et\v{r}il and R\"odl states that if $\mathcal{L}$ is a class of finite structures with both the JEP and the RP, then $\mathcal{L}$ also has the AP. While we are unable to prove that $\mathcal{K}'$ has the RP, we will use ideas from Ramsey theory to power the proof. Namely, if $\gamma_m$ is an unavoidable $R_m$-coloring of $H_m$ and we write $\gamma_m^{-1}(\{i\}) :=T = S_0\sqcup\cdots \sqcup S_{r-1}$, there is $\eta\in \widehat{G}$ and $i< r$ with $\eta^{-1}(S_i) = \eta^{-1}(T)$.
\vspace{3 mm}

\begin{theorem}
\label{RamseyAP}
Assume that $G$ is Roelcke precompact, and let $\mathcal{K}'$ be the big Ramsey expansion of $\mathcal{K}$. Then $\mathcal{K}'$ has the AP.
\end{theorem}

\begin{proof}
Fix a representation of $D_\omega$ based on $\{R_n: n< \omega\}$. Let $m\leq n< \omega$, $p, q\in R_n$, and $f_p, f_q\in H_m^n$ with $D_\omega(m,n)(p, f_p) = D_\omega(m,n)(q, f_q) = v$ for some $v\in R_m$. Using Roelcke precompactness, find $N\geq n$ large enough so that $\{\mathrm{tp}(i_n, s): s\in H_n^N\} = S_n^{(2)}$. Fix unavoidable colorings $\gamma:= \{\gamma_0\ll \cdots \ll \gamma_N\}$ so that $D_\gamma = D_\omega|_{N+1}$. Now consider the following partition of $T:= \gamma_m^{-1}(\{v\})$ into $4$ pieces $S_\emptyset$, $S_{\{p\}}$, $S_{\{q\}}$, and $S_{\{p,q\}}$. If $h\in T$ and $I\subseteq \{p, q\}$, we put $h\in S_I$ iff for $x\in \{p,q\}$, we have $x\in I$ iff there is $s\in H_n$ with $s\circ f_x = h$ and $\gamma_n(s) = x$. Fix some $I\subseteq \{p,q\}$ and $\eta\in \widehat{G}$ with $\eta^{-1}(S_I) = \eta^{-1}(T)$. But notice that $I$ must equal $\{p,q\}$ since both of the sets $\gamma_n^{-1}(\{p\})$ and $\gamma_n^{-1}(\{q\})$ are unavoidable. 

So fix $h\in T$ and $t_p, t_q\in H_n$ with $t_p\circ f_p = t_q\circ f_q = h$, $\gamma_n(t_p) = p$, and $\gamma_n(t_q) = q$. By choice of $N$, find $y\in H_N$ and $s_p, s_q\in H_n^N$ so that $y\circ s_p = t_p$ and $y\circ s_q = t_q$. Now setting $q' = \gamma_N(y)$, we are done.
\end{proof}
\vspace{3 mm}

While we are unable to prove in general that $\mathcal{K}'$ has the RP, we can show this holds if $\mathbf{K}$ admits a big Ramsey structure $\mathbf{K}'$. As a consequence, if there is a big Ramsey structure, then $\mathcal{K}'$ has the AP regardless of whether or not $G$ is Roelcke precompact by using the Ne\v{s}et\v{r}il--R\"odl theorem. Recall that with a big Ramsey structure, $\mathcal{K}'\cong \age{\mathbf{K}'/\mathcal{K}}$, and $\age{\mathbf{K}'/\mathcal{K}}$ trivially has the JEP.
\vspace{3 mm}

\begin{theorem}
\label{RPBR}
Suppose $\mathbf{K}$ admits a big Ramsey structure $\mathbf{K}'$. Then the big Ramsey expansion $\mathcal{K}' \cong \age{\mathbf{K}'/\mathcal{K}}$ has the RP. 
\end{theorem}

\begin{proof}
Fix $m \leq n <\omega$ and expansions $\mathbf{A}_n^* := \mathbf{K}'\cdot s$ and $\mathbf{A}_m^* := \mathbf{K}'\cdot (s\circ f)$ for some $s\in H_n$ and $f\in H_m^n$. By using some $g\in G$ and working with the big Ramsey structure $\mathbf{K}'\cdot g$ instead, we may assume that $s\circ f = i_m$; by passing to a possibly larger $n$, we may assume $f = i_m$ and $s = i_n$. It is enough (see section 4 of \cite{Z}) to show that $\mathbf{K}'\rightarrow (\mathbf{A}_n^*)^{\mathbf{A}_m^*}_2$, so fix a coloring $\gamma: \emb{\mathbf{A}_m^*, \mathbf{K}'}\rightarrow 2$. Notice that $\emb{\mathbf{A}_m^*, \mathbf{K}'} = \{h\in H_m: \mathbf{K}'\cdot h = \mathbf{A}_m^*\}$, and write $\emb{\mathbf{A}_m^*, \mathbf{K}'} := T = S_0\sqcup S_1$ for the two color classes. Find $\eta\in \widehat{G}$ and $i< 2$ with $\eta^{-1}(T) = \eta^{-1}(S_i)$. Find $s\in H_n$ with $\mathbf{K}'\cdot (\eta\cdot s) = \mathbf{A}_n^*$. Then $\eta\cdot s\in \emb{\mathbf{A}_n^*, \mathbf{K}'}$ satisfies that $|\{\gamma((\eta\cdot s)\circ f): f\in \emb{\mathbf{A}_m^*, \mathbf{A}_n^*}\}| = 1$ as desired.
\end{proof}
\vspace{3 mm}

\begin{cor}
\label{APBR}
In the setting of Theorem \ref{RPBR}, $\mathcal{K}'$ has the AP.
\end{cor}
\vspace{3 mm}

If $\mathcal{K}^*$ is a reasonable expansion of $\mathcal{K}$ in a language $L^*$, we can form the space $X_{\mathcal{K}^*}$ of $L^*$-structures $\mathbf{K}^*$ with underlying set $K$ such that $\age{\mathbf{K}^*/\mathcal{K}}\subseteq \mathcal{K}^*$; we endow $X_{\mathcal{K}^*}$ with the logic topology. If $\mathcal{K}^*$ is also precompact, then $X_{\mathcal{K}^*}$ is compact, hence a $G$-flow. If $\mathcal{K}'$ is the big Ramsey expansion of $\mathcal{K}$ and $x\in X_{\mathcal{K}'}$ is any completion point, then $x$ is a big Ramsey structure, and the ambit $(X_{\mathcal{K}'}, x)$ is a big Ramsey ambit. 

More generally, if $\mathcal{K}^*$ is reasonable, precompact, and has the JEP, then any $x\in X_{\mathcal{K}^*}$ with $\age{x/\mathcal{K}} = \mathcal{K}^*$ has dense orbit, so $X_{\mathcal{K}^*}$ is a pre-ambit. Furthermore, assume that $\mathcal{K}^*$ also has the AP, and let $\mathbf{F}^*\in X_{\mathcal{K}^*}$ be a \fr limit. Another useful property of structures equivalent to being a \fr structure is the \emph{Extension Property}. In our non-hereditary context, this reads as follows. 
\vspace{3 mm}

\begin{defin}
A structure $\mathbf{F}^*\in X_{\mathcal{K}^*}$ has the \emph{Extension Property} (EP) if for any $\mathbf{A}\subseteq \mathbf{B}$ with $\mathbf{A}, \mathbf{B}\in \mathcal{K}$, expansions $\mathbf{A}^*\subseteq \mathbf{B}^*$, and any embedding $f: \mathbf{A}^*\rightarrow \mathbf{F}^*$, there is an embedding $h: \mathbf{B}^*\rightarrow \mathbf{F}^*$ with $h|_A = f$.
\end{defin}
\vspace{3 mm}

This formulation has two important corollaries which we now describe. First note that orbits of $X_{\mathcal{K}^*}$ correspond exactly to isomorphism classes of structures, an any isomorphism between structures in $X_{\mathcal{K}^*}$ must also be an automorphism of $\mathbf{K}$. Now a structure $\mathbf{F}^*\in X_{\mathcal{K}^*}$ is isomorphic to the \fr limit iff $\mathbf{F}^*$ satisfies the EP, and this can easily be phrased as a countable intersection of open conditions. Therefore the orbit of $\mathbf{F}^*$ is $G_\delta$, and since $\age{\mathbf{F}^*/\mathbf{K}} = \mathcal{K}^*$, the orbit is also dense, therefore comeager. 

The second consequence of the EP is that starting from the \fr limit $\mathbf{F}^*$, it is ``easy'' to obtain any other structure in $X_{\mathcal{K}^*}$. Namely, if $\mathbf{K}^*\in X_{\mathcal{K}^*}$ is any structure, we can repeatedly use the extension property to find $\eta\in \widehat{G}$ with $\mathbf{F}^*\cdot \eta = \mathbf{K}^*$.

These two results conspire to make the search for completion points of $X_{\mathcal{K}'}$ difficult. Corollary \ref{APBR} tells us that when $G$ is Roelcke precompact, $\mathcal{K}'$ admits a \fr limit $\mathbf{F}'\in X_{\mathcal{K}^*}$, and the orbit of $\mathbf{F}'$ in $X_{\mathcal{K}'}$ is comeager. The following proposition shows that in most circumstances, the generic orbit is the wrong orbit to investigate. Recall from the introduction the definition of \emph{small Ramsey degree}. It is easy to see that since each $\mathbf{A}_m$ has finite big Ramsey degree $R_m$, then each $\mathbf{A}_m$ also has finite small Ramsey degree $r_m\leq R_m$. The key consequence of having small Ramsey degree $r_m$ that we will need is as follows. If $\mathcal{K}^*$ is any reasonable, precompact expansion of $\mathcal{K}$, then there is $\mathbf{K}^*\in X_{\mathcal{K}^*}$ with $|\{\mathbf{K}^*\cdot f: f\in H_m\}| \leq r_m$.
\vspace{3 mm}

\begin{prop}
\label{GenericNotBR}
For each $m< \omega$, let $r_m$ be the small Ramsey degree of $\mathbf{A}_m$. Suppose for some $m< \omega$ that $r_m< R_m$. Then if $G$ is Roelcke precompact and $\mathbf{F}' = \flim{\mathcal{K}'}\in X_{\mathcal{K}'}$, then $\mathbf{F}'$ is not a big Ramsey structure.
\end{prop}

\begin{proof}
As $r_m$ is the small Ramsey degree of $\mathbf{A}_m$, there is some $\mathbf{K}^*\in X_{\mathcal{K}'}$ with $|\{\mathbf{K}^*\cdot f: f\in H_m\}| \leq r_m < R_m$. In particular, such a $\mathbf{K}^*$ is not a big Ramsey structure. Since $\mathbf{K}^* = \mathbf{F}'\cdot \eta$ for some $\eta\in \widehat{G}$, it follows that $\mathbf{F}'$ cannot be a big Ramsey structure.
\end{proof}
\vspace{3 mm}

\begin{rem}
If $r_m = R_m$ for every $m< \omega$, then we can find for every $m< \omega$ a \emph{syndetic} $r_m$-coloring $\gamma_m$ of $H_m$ with $\gamma_m\ll \gamma_n$ for every $m\leq n< \omega$. Such a sequence of colorings can then be used to construct the universal minimal flow $M(G)$ of $G$ (see section 8 of \cite{Z}). As syndetic colorings are unavoidable, it follows that $X_{\mathcal{K}'}$ is just $M(G)$; as every orbit is dense, every point is a completion point, so $X_{\mathcal{K}'}$ is a big Ramsey flow. A good example to keep in mind is when $\mathbf{K} = \bigcup_n \mathbf{A}_n$ is a countable set $K$ with no structure, and each $\mathbf{A}_m$ is a set of size $m$ with no structure. Then $r_m = R_m = m!$, $\mathcal{K}'$ is the class of finite linear orders, and $X_{\mathcal{K}'}$ is the space of linear orders on $K$. 
\end{rem}

\section{Connections and questions}
\label{QuestionSection}

This section discusses the connections between completion flows and the notion of \emph{oscillation stability} from \cite{KPT}. It also gathers a list of open questions.

In this section, the only uniform structure we will consider on a topological group $G$ is the left uniformity, and any references to uniform continuity, Cauchy, etc.\ should be interpreted as such. If $f: G\rightarrow [0,1]$ is uniformly continuous, then $f$ continuously extends to the left completion $\widehat{G}$, and we will also use $f$ to denote this extension.
\vspace{3 mm}

\begin{defin}\mbox{}
\begin{enumerate}
\item
Let $G$ be a Hausdorff topological group with left completion $\widehat{G}$. A left-uniformly continuous function $f: G\rightarrow [0,1]$ is called \emph{oscillation stable} if for every $\eta\in \widehat{G}$ and $\epsilon > 0$, there is $\zeta\in \widehat{G}$ so that $\sup(|f\circ \eta\circ \zeta(g) - f\circ \eta\circ \zeta(h)|: g, h\in G)< \epsilon$ 

\item
The topological group $G$ is called \emph{oscillation stable} if every left-uniformly continuous $f: G\rightarrow [0,1]$ is oscillation stable.
\end{enumerate}
\end{defin}
\vspace{3 mm}

Suppose $(X, x_0)$ is a non-trivial completion ambit, and let $f: X\rightarrow [0,1]$ be any continuous non-constant function. Then the function $f(x_0-): G\rightarrow [0,1]$ cannot be oscillation stable. However, given a continuous function $f: G\rightarrow [0,1]$ which is not oscillation stable, it is not clear that $f$ can be described in this fashion. 
Recall that bounded uniformly continuous functions are precisely those functions which extend continuously to $S(G)$.
\vspace{3 mm}

\begin{prop}
\label{CompleteFunction}
Let $G$ be a topological group. The following are equivalent.
\begin{enumerate}
\item
$G$ admits a non-trivial completion ambit.
\item
There is a non-constant, uniformly continuous $f: G\rightarrow [0,1]$ so that for any $\eta\in \widehat{G}$, there are $g_i\in G$  so that $f\circ \eta \circ g_i \rightarrow f$ pointwise.
\end{enumerate}
\end{prop}

\begin{rem}
The pointwise convergence in (2) is only for the functions with domain $G$. In general, we cannot get pointwise convergence on all of $S(G)$.
\end{rem}

\begin{proof}
Suppose $(X, x_0)$ is a non-trivial completion ambit. Let $\phi: X\rightarrow [0,1]$ be continuous, and define $f: G\rightarrow [0,1]$ via $f(g) = \phi(x_0g)$. If $\eta\in \widehat{G}$, find $g_i\in G$ with $x_0\eta g_i \to x_0$. Then $f\circ \eta\circ g_i(g) = f(\eta\circ g_i\circ g) = \phi(x_0\eta g_i g) = \phi(x_0 g) = f(g)$.

Suppose $f: G\rightarrow [0,1]$ satisfies (2). View $f$ as a member of $[0,1]^G$ equipped with the product topology. $G$ acts on this space by right shift. Let $X\subseteq [0,1]^G$ denote the orbit closure of $f$, and notice that every member of $X$ is uniformly continuous. To be extra careful, let us write $(f\cdot p)_X$ to denote various dynamical computations as carried out in $X$. Notice that for any $\eta\in \widehat{G}$, we have
\begin{align*}
(f\cdot \eta)_X(g) := \lim_{g_i\to \eta} f\cdot g_i(g) = \lim_{g_i\to \eta} f(g_ig) = f(\eta g)
\end{align*}
Find $g_i\in G$ with $f\circ \eta \circ g_i\rightarrow f$ pointwise. But $f\circ \eta\circ g_i = (f\cdot \eta\cdot g_i)_X$, so $f$ is a completion point of $X$.
\end{proof}
\vspace{3 mm}

If $G = \aut{\mathbf{K}} = \bigcup_n \mathbf{A}_n$, then any uniformly continuous function $f: G\rightarrow [0,1]$ can be uniformly approximated by functions of the form $f'\circ \pi_m$ for some $f': H_m\rightarrow [0,1]$. It follows that $G$ is oscillation stable iff every $\mathbf{A}_n$ has big Ramsey degree $1$, which cannot happen.

It is unknown whether any non-trivial oscillation stable topological groups exist. Hjorth has shown \cite{Hj} that no Polish group can be oscillation stable; a simpler proof of this result is given by Melleray \cite{Mel}. As indicated in the introduction, let us propose the following weakening of oscillation stability.
\vspace{3 mm}

\begin{defin}
Let $G$ be a Hausdorff topological group. We call $G$ \emph{completely amenable} if $G$ admits no non-trivial completion flows.
\end{defin}
\vspace{3 mm}

Every oscillation stable group is completely amenable, and every completely amenable group is extremely amenable. This brings us to our first question.
\vspace{3 mm}

\begin{que}
\label{CompleteQue}
Are there non-trivial topological groups which are completely amenable? Are any non-trivial Polish groups completely amenable? Are any non-trivial groups $G = \aut{\mathbf{K}}$ completely amenable?
\end{que}
\vspace{3 mm}

If a topological group $G$ does admit a non-trivial completion flow, then we can ask about the structure of the collection of completion flows and surjective $G$-maps. We have seen that some groups admit non-trivial universal completion flows which are unique up to isomorphism.
\vspace{3 mm}

\begin{que}
\label{UCFQue}
Let $G$ be a topological group. Does $G$ admit a universal completion flow? If $G$ does admit a universal completion flow, is it unique?
\end{que}
\vspace{3 mm}

In the case that $G = \aut{\mathbf{K}}$, where $\mathbf{K} = \bigcup_n \mathbf{A}_n$ where each $\mathbf{A}_n$ has finite big Ramsey degree, we constructed in section 6 the big Ramsey expansion class $\mathcal{K}'$. If $G$ is also \emph{Roelcke precompact}, we were able to show that $\mathcal{K}'$ has the JEP and the AP.
\vspace{3 mm}

\begin{que}
\label{RoelckeQue}
Let $G = \aut{\mathbf{K}}$, where $\mathbf{K} = \bigcup_n \mathbf{A}_n = \flim{\mathcal{K}}$ and each $\mathbf{A}_n$ has finite big Ramsey degree. Let $\mathcal{K}'$ be the big Ramsey expansion class of $\mathcal{K}$. Is there a ``dynamical'' characterization of $X_{\mathcal{K}'}$? Is $X_{\mathcal{K}'}$ a big Ramsey flow? Is this true if $G$ is Roelcke precompact?
\end{que}
\vspace{3 mm}

\begin{que}
\label{RamseyMetrQue}
Let $G = \aut{\mathbf{K}}$, where $\mathbf{K} = \flim{\mathcal{K}}$. Assume some $\mathbf{A}\in \mathcal{K}$ does not have finite big Ramsey degree. Then does $G$ admit a non-metrizable completion flow?
\end{que}
\vspace{3 mm}

We now turn to more specific questions. First suppose that $G$ and $H$ are topological groups and $\phi: G\rightarrow H$ is a continuous homomorphism with dense image. Then $\phi$ is uniformly continuous when both $G$ and $H$ are given their left uniform structures, so we may extend $\phi$ to a map from $\widehat{G}$ to $\widehat{H}$, which we also denote by $\phi$. Now suppose that $(X, x_0)$ is an $H$-completion-ambit. We may regard $X$ as a $G$-flow by setting $x\cdot g = x\cdot \phi(g)$. Since $\phi$ has dense image, $x_0$ still has dense orbit, and since $\phi$ maps $\widehat{G}$ to $\widehat{H}$, $(X,x_0)$ is also a $G$-completion-ambit. 

We have seen that if $G = \aut{\mathbb{Q}}$, then $G$ admits a unique universal completion flow, which furthermore is metrizable. Let $H = \mathrm{Homeo}^+([0,1])$ be the group of orientation-preserving homeomorphisms of the unit interval; we endow $H$ with the compact-open topology. A compatible left-invariant metric is given as follows. If $h_0, h_1\in H$, set $d(h_0, h_1) = \max(|h_0^{-1}(x)- h_1^{-1}(x)|: x\in [0,1])$. Fix an order-preserving injection $f: \mathbb{Q}\rightarrow [0,1]$ with dense image, and use $f$ to obtain a continuous homomorphism $\phi: G\rightarrow H$ with dense image. It follows that every $H$-completion-flow is metrizable. The following question in some sense asks if a ``continuous'' analogue of Devlin's theorem holds. 
\vspace{3 mm}

\begin{que}
\label{HomeoQue}
Let $H = \mathrm{Homeo}^+([0,1])$ be the group of orientation preserving homeomorphisms of the unit interval with the compact-open topology. Is $H$ completely amenable? Does $H$ admit a unique universal completion flow?
\end{que}
\vspace{3 mm}

The final question investigates the possibility of ``iterated'' Ramsey degrees in the following sense. Consider $G = \aut{\mathbb{Q}}$, and let $\mathbf{Q}'$ be a big Ramsey structure for $\mathbb{Q}$. Set $\mathcal{K}' = \age{\mathbf{Q}'/\mathcal{K}}$. We saw in section 6 that $\mathcal{K}'$ is a Fra\"iss\'e--HP class, so has a \fr limit $\mathbf{F}'\in X_{\mathcal{K}'}$.

\begin{que}
\label{IteratedQue}
With notation as in the previous paragraph, does $\mathbf{F}'$ admit a big Ramsey structure?
\end{que}

\appendix
\section{Proof of Theorem \ref{LevelsBR}}
\label{ProofAppendix}

We restate Theorem \ref{LevelsBR} below.

\begin{theorem*}
Let $\mathbf{K} = \bigcup_n \mathbf{A}_n$ be a \fr structure, and suppose each $\mathbf{A}_m$ has finite big Ramsey degree $R_m < \omega$. Assume that for each $m <\omega$, there is an unavoidable $R_m$-coloring $\gamma_m$ of $H_m$ so that $\gamma_m \ll \gamma_n$ for each $m\leq n < \omega$. Then $\mathbf{K}$ admits a big Ramsey structure.
\end{theorem*}

We return to our convention that $\mathcal{K} = \age{\mathbf{K}}$; indeed, it was this proposition which justified the ``Fra\"iss\'e--HP'' perspective we took in later sections. In particular, we will need to deal with finite structures $\mathbf{B}\subseteq \mathbf{K}$ not equal to some $\mathbf{A}_n$. Many of the definitions and theorems from section \ref{SetsAndColorings} generalize to deal with any finite structures in $\mathcal{K}$, and we will freely use the ``extended'' versions of these theorems.

We will need the following easy lemma, which is very similar to Proposition \ref{BRCompletion}
\vspace{3 mm}

\begin{lemma}
\label{ColorCompletion}
With $\gamma_n$ as in the statement of Theorem \ref{LevelsBR}, then if $\eta\in \widehat{G}$, there is $p\in S(G)$ with $\gamma_n\cdot \eta\cdot p = \gamma_n$.
\end{lemma}

\begin{proof}
For each $N\geq n$, find $g_N\in G$ so that $\gamma_N\cdot \eta\cdot g_N$ and $\gamma_N$ agree on $i_N$. Let $p\in S(G)$ be a cluster point of the $g_N$. Since $\gamma_n \ll \gamma_N$ for every $n\leq N < \omega$, we have $\gamma_n\cdot \eta\cdot p = \gamma_n$ as desired.
\end{proof}
\vspace{3 mm}

\begin{proof}[Proof of Proposition \ref{LevelsBR}]
First notice that each $\mathbf{B}\in \mathcal{K}$ has finite big Ramsey degree, as for some $n< \omega$, we have $\mathbf{B}\leq \mathbf{A}_n$. Let $R_\mathbf{B} < \omega$ be the big Ramsey degree of $\mathbf{B}\in \mathcal{K}$.

We produce for every $\mathbf{B}\in \mathcal{K}$ with $\mathbf{B}\subseteq \mathbf{K}$ a coloring $\gamma_\mathbf{B}$ of $H_\mathbf{B}:= \emb{\mathbf{B}, \mathbf{K}}$ so that the following items hold.
\begin{enumerate}
\item
Each $\gamma_\mathbf{B}$ is an unavoidable $R_\mathbf{B}$-coloring.

\item
If $\mathbf{B}\leq \mathbf{C}\in \mathcal{K}$ and $f\in \emb{\mathbf{B}, \mathbf{C}}$, then $\gamma_\mathbf{C}$ refines $\gamma_\mathbf{B}\circ \hatf$ (i.e.\ $\gamma_\mathbf{B}\ll \gamma_\mathbf{C}$).
\end{enumerate}

Let us show how to complete the proof given these colorings. Suppose $\mathbf{K}$ is an $L$-stucture. We produce a new language $L'\supseteq L$; for each $\mathbf{B}\in \mathcal{K}$ with $\mathbf{B}\subseteq \mathbf{K}$, we introduce new relational symbols $\{S(\mathbf{B}, 0),...,S(\mathbf{B}, R_\mathbf{B}-1)\}$ of arity $|B|$. We now construct an $L'$-structure $\mathbf{K}'$ on the underlying set $K$. If $R\in L$, we set $R^{\mathbf{K}'} = R^\mathbf{K}$. To interpret the new relational symbols, first fix for each $\mathbf{B}\in \mathcal{K}$ with $\mathbf{B}\subseteq \mathbf{K}$ an enumeration of the underlying set $B = \{b_0,...,b_{k-1}\}$, where $k = |B|$. Then if $a_0,...,a_{k-1}$ is a $k$-tuple from $K$ and $j< R_\mathbf{B}$, we set $S(\mathbf{B}, j)^{\mathbf{K}'}(a_0,...,a_{|B|-1})$ iff there is $f\in H_\mathbf{B}$ with $f(b_i) = a_i$ for each $i< k$ and $\gamma_\mathbf{B}(f) = j$. 

Given $\mathbf{B}\in \mathcal{K}$ and $\mathbf{K}'$ as constructed above, recall that $\mathbf{K}'(\mathbf{B}) = \{\mathbf{K}'\cdot f: f\in H_\mathbf{B}\}$. In order to show that $\mathbf{K}'$ as constructed above is a big Ramsey structure, it is enough to show that every $\mathbf{B}\in \mathcal{K}$ satisfies $|\mathbf{K}'(\mathbf{B})| = R_\mathbf{B}$. We may assume $\mathbf{B}\subseteq \mathbf{K}$; let $B = \{b_0,...,b_{k-1}\}$ be the enumeration used in the construction of $\mathbf{K}'$. Notice that if $\mathbf{B}' \in \mathbf{K}'(\mathbf{B})$, then for some unique $j< R_\mathbf{B}$, we have $S(\mathbf{B}, j)^{\mathbf{K}'}(b_0,...,b_{k-1})$. Call such a $\mathbf{B}'$ an \emph{expansion of type $j$}. Now suppose $\mathbf{B}', \mathbf{B}^*\in \mathbf{K}'(\mathbf{B})$ are both expansions of type $j$. We will show that $\mathbf{B}' = \mathbf{B}^*$. Fix $f', f^*\in H_\mathbf{B}$ so that $\mathbf{B}' = \mathbf{K}'\cdot f'$ and $\mathbf{B}^* = \mathbf{K}'\cdot f^*$. It is enough to show that for any $\mathbf{A}\in \mathcal{K}$ with $\mathbf{A}\subseteq \mathbf{K}$, with $A = \{a_0,...,a_{r-1}\}$, any $r$-tuple $c_0,...,c_{r-1}\in B$, and any $\ell< R_\mathbf{A}$ that $S(\mathbf{A}, \ell)^{\mathbf{B}'}(c_0,...,c_{r-1})$ holds iff $S(\mathbf{A}, \ell)^{\mathbf{B}^*}(c_0,...,c_{r-1})$ holds. 

By symmetry, it is enough to show one implication, so suppose $S(\mathbf{A}, \ell)^{\mathbf{B}'}(c_0,...,c_{r-1})$ holds. By definition, this means that $S(\mathbf{A}, \ell)^{\mathbf{K}'}(f'(c_0),...,f'(c_{r-1}))$ holds. In particular, the map $h: A\rightarrow K$ given by $h(a_i) = f'(c_i)$ for $i< r$ is an embedding of $\mathbf{A}$ into $\mathbf{K}$, so the map $s: A\rightarrow B$ given by $s(a_i) = c_i$ is an embedding of $\mathbf{A}$ into $\mathbf{B}$. Notice that $f'\circ s = h$. By item (2), we must have that $\gamma_\mathbf{B}$ refines $\gamma_\mathbf{A}\circ \hat{s}$. Since $\gamma_\mathbf{B}(f^*) = \gamma_\mathbf{B}(f') = j$, we must have $\gamma_\mathbf{A}(f^*\circ s) = \gamma_\mathbf{A}(f'\circ s) = \gamma_\mathbf{A}(h) = \ell$. It follows that $S(\mathbf{A}, \ell)^{\mathbf{K}'}(f^*(c_0),...,f^*(c_{r-1}))$ holds, so also $S(\mathbf{A}, j)^{\mathbf{B}^*}(c_0,...,c_{r-1})$ holds.
\vspace{3 mm}

We now proceed to construct the colorings $\gamma_\mathbf{B}$ satisfying items (1) and (2) above. Fix $\mathbf{B}\in \mathcal{K}$ with $\mathbf{B}\subseteq \mathbf{K}$, and find $n< \omega$ large enough so that $\mathbf{B}\subseteq \mathbf{A}_n$. Let $i_\mathbf{B}: \mathbf{B}\rightarrow \mathbf{A}_n$ denote the inclusion embedding. For each $f\in H_\mathbf{B}$, let $S_f = \{j < R_n: \exists s\in H_n (s\circ i_\mathbf{B} = f \text{ and } \gamma_n(s) = j)\}$. We define a reflexive graph $\Gamma$ on $H_\mathbf{B}$ by declaring $(f, h)\in \Gamma$ iff $S_f\cap S_h\neq \emptyset$. Define a coloring $\gamma_\mathbf{B}$ on $H_\mathbf{B}$ by sending $f\in H_\mathbf{B}$ to the connected component of $f$ in $\Gamma$.

First let us argue that $\gamma_\mathbf{B}$ is an unavoidable coloring. Fix a connected component $X\subseteq H_\mathbf{B}$ of $\Gamma$. We can write $X = \{f\in H_\mathbf{B}: S_f\subseteq S\}$ for some $S\subseteq R_n$, namely $S = \bigcup_{f\in X} S_f$. But then we also have $X = \{f\in H_\mathbf{B}: S_f\cap S\neq \emptyset\}$. Now fix $\eta\in \widehat{G}$ towards showing that $\eta^{-1}(X)\neq \emptyset$. Pick $j\in S$, and find $s\in H_n$ with $\gamma_n(\eta\circ s) = j$. Then $\eta\circ s\circ i_\mathbf{B}\in X$, so $\eta^{-1}(X)\neq\emptyset$.

To see that $\gamma_\mathbf{B}$ is an unavoidable $R_\mathbf{B}$-coloring, let $\delta: H_\mathbf{B}\rightarrow R_\mathbf{B}$ be unavoidable. Find $\eta\in \widehat{G}$ with $\gamma_\mathbf{B}\cdot \eta \leq \delta\cdot \eta \ll \gamma_n\cdot \eta$. Using Lemma \ref{ColorCompletion}, find $p\in S(G)$ with $\gamma_n\cdot \eta \cdot p = \gamma_n$. By Lemma \ref{RefineUnavoidable}, we have $\delta\cdot \eta \cdot p$ an unavoidable $R_\mathbf{B}$-coloring with $\delta \cdot \eta\cdot p \ll \gamma_n$. We will show that $\gamma_\mathbf{B} \sim \delta\cdot \eta\cdot p$. We must have $\gamma_\mathbf{B} \geq \delta\cdot \eta\cdot p$; by construction, $\gamma_{\mathbf{B}}$ is the finest possible coloring with $\gamma_\mathbf{B} \ll \gamma_n$. But since $\gamma_B$ is an unavoidable coloring and $\delta\cdot \eta\cdot p$ is an unavoidable $R_\mathbf{B}$-coloring, we must have $\gamma_\mathbf{B}\sim \delta\cdot\eta \cdot p$. In particular, $\gamma_\mathbf{B}$ is an unavoidable $R_\mathbf{B}$-coloring as desired.
\end{proof}

Andy Zucker

Carnegie Mellon University

Pittsburgh, PA 15213

andrewz@andrew.cmu.edu

\end{document}